\documentclass[10pt,twoside,fleqn]{article}
\pdftrailerid{}
\usepackage[utf8]{inputenc}
\usepackage[T1]{fontenc}
\usepackage{amsmath,amssymb,dsfont}
\usepackage{mathtools}
\numberwithin{equation}{section}
\usepackage{microtype}
\usepackage{xcolor}
\usepackage[hyperref,amsmath,thmmarks]{ntheorem}
\usepackage{aliascnt}
\usepackage[a4paper,centering,bindingoffset=0cm,marginpar=2cm,margin=2.5cm]{geometry}
\usepackage[pagestyles]{titlesec}
\usepackage[font=footnotesize,format=plain,labelfont=sc,textfont=sl,width=0.75\textwidth,labelsep=period]{caption}
\usepackage{bm}
\usepackage{enumitem} 
\setlist[itemize]{leftmargin=*,topsep=0pt}

\usepackage[backend=biber,maxnames=10,backref=true,hyperref=true,giveninits=true,safeinputenc]{biblatex}
\DefineBibliographyStrings{english}{%
	backrefpage = {cited on page},
	backrefpages = {cited on pages},
}

\DeclareFieldFormat[report]{title}{``#1''}
\DeclareFieldFormat[book]{title}{``#1''}
\AtEveryBibitem{\clearfield{url}}
\AtEveryBibitem{\clearfield{note}}

\titleformat{\section}[block]{\large\sc\filcenter}{\thesection.}{0.5ex}{}[]
\titleformat{\subsection}[runin]{\bf}{\thesubsection.}{0.5ex}{}[.]

\usepackage[pdftex,colorlinks=true,linkcolor=blue,citecolor=green,urlcolor=blue,bookmarks=true,bookmarksnumbered=true]{hyperref}
\hypersetup
{
    bookmarksnumbered,
    breaklinks,
    pdfborderstyle=,
    colorlinks,
    citecolor=[rgb]{0,.65,0},
    linkcolor=[rgb]{0,0,.5},
    urlcolor=[rgb]{0,0,.5},
    pdfauthor={Authors},
    pdfsubject={Subject},
    pdftitle={Title},
    pdfkeywords={Keywords}
}

\newpagestyle{headers}
{
	\headrule
	\sethead[\footnotesize\thepage][\footnotesize\sc P.~Elbau and D.~Schmutz][]{}{\footnotesize\sc Uniqueness of angular velocity reconstructions}{\footnotesize\thepage}
	\setfoot{}{}{}
}
\theoremseparator{.}
\newtheorem{lemma}{Lemma}[section]

\newaliascnt{proposition}{lemma}
\newtheorem{proposition}[proposition]{Proposition}
\aliascntresetthe{proposition}

\newaliascnt{corollary}{lemma}
\newtheorem{corollary}[corollary]{Corollary}
\aliascntresetthe{corollary}

\newaliascnt{theorem}{lemma}
\newtheorem{theorem}[theorem]{Theorem}
\aliascntresetthe{theorem}

\theorembodyfont{\normalfont}
\newaliascnt{definition}{lemma}
\newtheorem{definition}[definition]{Definition}
\aliascntresetthe{definition}

\newaliascnt{assumption}{lemma}

\aliascntresetthe{assumption}

\newaliascnt{problem}{lemma}
\newtheorem{problem}[problem]{Problem}
\aliascntresetthe{problem}

\newaliascnt{example}{lemma}
\newtheorem{example}[example]{Example}
\aliascntresetthe{example}

\theoremseparator{:}
\theoremheaderfont{\normalfont\itshape}
\newaliascnt{remark}{lemma}
\newtheorem{remark}[remark]{Remark}
\aliascntresetthe{remark}

\theoremstyle{nonumberplain}
\theoremsymbol{\ensuremath{\square}}
\newtheorem{proof}{Proof}

\newcommand{\N}{\mathds{N}}

\newcommand{\R}{\mathds{R}}
\newcommand{\C}{\mathds{C}}

\newcommand{\I}{\mathds{I}}

\let\Re=\undefined
\DeclareMathOperator{\Re}{Re}

\let\Im=\undefined
\DeclareMathOperator{\Im}{Im}
\DeclareMathOperator{\supp}{supp}

\DeclareMathOperator{\sgn}{sgn}
\newcommand{\e}{\mathrm e}
\let\ii\i
\renewcommand{\i}{\mathrm i}
\renewcommand{\d}{\,\mathrm d}

\DeclareMathOperator{\spn}{span}

\def\curl{\operatorname{curl}}
\def\div{\operatorname{div}}
\def\grad{\operatorname{grad}}
\def\diag{\operatorname{diag}}
\def\D{\mathrm D}
\newcommand{\B}{\mathrm B}
\newcommand{\SO}{\mathrm{SO}}

\DeclarePairedDelimiter{\abs}{\lvert}{\rvert}
\DeclarePairedDelimiter{\norm}{\lVert}{\rVert}
\DeclarePairedDelimiterX\inner[2]{\langle}{\rangle}{#1,#2}
\DeclarePairedDelimiterX\set[2]{\lbrace}{\rbrace}{#1\;\delimsize\vert\;\mathopen{}#2}

\title{Uniqueness of Angular Velocity Reconstruction\\in Parallel-Beam and Diffraction Tomography}
\author{Peter Elbau\\{\footnotesize\href{mailto:peter.elbau@univie.ac.at}{peter.elbau@univie.ac.at}}
\and Denise Schmutz\\{\footnotesize\href{mailto:denise.schmutz@univie.ac.at}{denise.schmutz@univie.ac.at}}}
\date{}

\begin{document}

\maketitle
\thispagestyle{empty}
\begin{center}
\newlength{\lengthParbox}
\settowidth{\lengthParbox}{\footnotesize Oskar-Morgenstern-Platz 1}
\parbox[t]{\lengthParbox}{\footnotesize
Faculty of Mathematics\\
University of Vienna\\
Oskar-Morgenstern-Platz 1\\
A-1090 Vienna, Austria}
\end{center}

\bigskip

\begin{abstract}
	This work addresses the problem of uniquely determining a rotational motion from continuous time-dependent measurements within the frameworks of parallel-beam and diffraction tomography. The motivation stems from the challenge of imaging trapped biological samples manipulated and rotated using optical or acoustic tweezers. We analyze the conditions under which the rotation of the unknown sample can be uniquely recovered using the infinitesimal common line and circle method, respectively. We provide explicit criteria for the sample’s structure and the induced motion that guarantee unique reconstruction of all rotation parameters. Moreover, we demonstrate that the set of objects for which uniqueness fails is nowhere dense.
\end{abstract}

\section{Introduction}
In tomographic imaging methods, such as computed tomography, it is typical to perform measurements of the sample from different directions to gather enough information for a full reconstruction of the object. This idea relies, however, on the crucial assumption that the sample remains completely unchanged during these repeated recordings, which is in practice not always easy to guarantee.

This led in \cite{SchmLou02,SchmLouWolVau02} to the development of the field of dynamic inverse problems which attempts to compensate for such unavoidable motions during the measurements. Even in the optimal case where the deformations of the object are fully known, it is not always possible to retain a perfect reconstruction, which is described in the articles \cite{HahQui16,HahGarQui21}. In most cases, however, the deformations are unknown so that we face the additional complication of having to determine the transformations from the data, see \cite{LuMac02,KatSilZam11}, for example. The considered deformations typically still need to be restricted to a certain class of mappings to be able to achieve a reconstruction.
In this article, we want to analyze one of the probably simplest classes of such problems, namely those where we only allow for global rotations of the object.

A very famous example for this setting is the single-particle analysis in cryogenic electron microscopy, which took its beginning with the seminal paper \cite{AdrDubLepMcd84}. The data hereby consists of a collection of transmission electron microscope images, which can be seen as values of the $X$-ray transform in the incident direction of the electron beam of the absorption coefficient of (multiple copies of) one object in various unknown orientation states. The main observation that allowed the reconstruction of the rotational state of the object was that there exists for every pair of two-dimensional Fourier transforms of the recorded images in each of them one straight line such that the values along these lines are equal. This led to the so-called common line method which had its origins in the articles \cite{CroDerKlu70,Hee87,Gon88a}.

However, this method cannot reproduce the rotational state in all cases, since it is for example clearly impossible to detect the orientation of a spherically symmetrical object, so that the question arises on how to characterize those objects for which a reconstruction of the orientation states is possible. It was shown in the article \cite{Lam08}, by using the moment method for the reconstruction, as introduced in \cite{Gon88a}, that this is possible for a generic object, meaning that the set of samples whose orientation cannot be detected is nowhere dense.
Unlike the common line method, the moment method is also applicable to parallel-beam tomography in two dimensions and for this case, there also exist the results in \cite{BasBre00,LamYli07}, which use the Helgason–Ludwig consistency conditions for the Radon transform to characterize such objects in terms of their moments. Another result in this direction which takes the wave properties of the imaging electrons into account has been given in \cite{KurZic21_report}. 

While this is intrinsically a problem with a discrete set of transformation states, we want to consider here the case where the object is continuously rotating during the recording of tomographic data. As a concrete experimental example, we have the measuring of microscopy data of trapped particles in mind. In this imaging technique a biological sample is illuminated while being held and rotated by optical or acoustical forces, as described in the paper \cite{LovPreThaRit21}. Since the forces acting on the sample depend on its unknown internal properties, its rotation can hereby not be perfectly controlled and we therefore would like to reconstruct it from the measurements along with the refractive index of the sample. In previous works, we examined two approximating models for the light propagation in this problem, each applicable depending on the available measurements and properties of the sample, and developed methods for the reconstruction of the motion which try to make use of the smooth motion: In \cite{ElbRitSchSchm20}, we modeled it as a parallel-beam tomography experiment (so that we are in the same setting as in the example from cryogenic electron microscopy) and derived the infinitesimal common line method; and in \cite{QueElbSchSte24}, we employed a diffraction tomography model and devised the infinitesimal common circle method to obtain the desired reconstruction of the rotational motion. (We will recall these methods briefly in \autoref{section:dt} and \autoref{section:pb}.)

In this article, we want to study now under which precise conditions these two reconstruction algorithms for the parallel-beam and the diffraction tomography model are guaranteed to recover the correct rotational motion and show that the necessary assumptions are indeed fulfilled for generic objects. In this case, the motion reconstruction can be used as a preprocessing step for the reconstruction of the refraction index of the object, since, once the motion is determined, the problem reduces to the standard inverse problem for parallel-beam and diffraction tomography, respectively, for which we refer to the textbooks \cite{KakSla01,NatWub01}.

Although the ideas to show the reconstructability of the motion are similar in both models, the reconstruction algorithm and the necessary criteria the objects should fulfill for their motions to be recoverable are still quite different so that the article is essentially split into two parts:
\begin{itemize}
	\item \autoref{section:dt} focuses on diffraction tomography under the Born approximation. In \autoref{def:dt_symmetry}, we introduce the concepts of DT-symmetry and DT-asymmetry and then prove in \autoref{thm:dt_uniqueness} that the rotational motion can be uniquely reconstructed for DT-asymmetric objects with the infinitesimal common circle method. Furthermore, we show in \autoref{thm:dt_S_nowheredense} that the set of DT-symmetric objects is nowhere dense so that motion reconstruction is feasible for generic samples.
	\item \autoref{section:pb} addresses the model of parallel-beam tomography. Unlike in the case of diffraction tomography, certain rotational motions must be excluded, as the method cannot be applied to them: We describe these degenerate motions in \autoref{def:pb_degmotion}. In \autoref{def:pb-symmetry}, we define PB-symmetry and PB-asymmetry and demonstrate in \autoref{thm:pb_uniqueness} that the rotational motion can be uniquely reconstructed for PB-asymmetric objects via the infinitesimal common line method. Finally, we show in \autoref{thm:pb_S_nowheredense} that the set of PB-symmetric objects is nowhere dense, indicating that rotation reconstruction is achievable for generic samples as in the case of the diffraction tomography model.
\end{itemize}

\subsection{Experimental setup}\label{subsection:experimental-setup}

To motivate the two main problems considered in this article, namely \autoref{pr:dt} and \autoref{pr:pb}, we give here a simplified mathematical model of the measurement setup described in the article~\cite{LovPreThaRit21} which will lead us to exactly these sorts of reconstruction problems.

We illuminate in this setting a sample, which we can rotate freely around the origin, by a light beam and detect the scattered light on a plane behind the object. We describe the actual orientation of the sample by a rotation matrix $R\in \SO(3)$ in such a way that the object is obtained by rotating it from a certain reference state according to $R^{-1}$.

We employ a classical scattering model, to be found in the textbook \cite[Chapter XIII]{BorWol99}, for example, where we describe the light as an electromagnetic wave. The incident beam is then given by an electric field $E^{(0)}\colon\R\times\R^3\to\R^3$ and a magnetic field $B^{(0)}\colon\R\times\R^3\to\R^3$ which solve Maxwell's equations
\begin{equation}\label{eq:Maxwell_vacuum}
\begin{aligned}
\frac1c\partial_\tau E^{(0)}(\tau,x) &= \curl B^{(0)}(\tau,x),&\div E^{(0)}(\tau,x)&=0, \\
\frac1c\partial_\tau B^{(0)}(\tau,x) &= -\curl E^{(0)}(\tau,x),&\div B^{(0)}(\tau,x)&=0
\end{aligned}
\end{equation}
for all $(\tau,x)\in\R\times\R^3$. The parameter $c$ hereby denotes the speed of light in the vacuum and the vector operators are only acting on the spatial variable $x$.

Assuming that we have no free charges, the resulting electric field $E_R\colon\R\times\R^3\to\R^3$, the displacement field $D_R\colon\R\times\R^3\to\R^3$, the magnetic field $B_R\colon\R\times\R^3\to\R^3$, and the magnetizing field $H_R\colon\R\times\R^3\to\R^3$ in the presence of the sample are then a solution of Maxwell's macroscopic equations
\begin{equation}\label{eq:Maxwell}
\begin{aligned}
\frac1c\partial_\tau D_R(\tau,x) &= \curl H_R(\tau,x),&\div D_R(\tau,x)&=0, \\
\frac1c\partial_\tau B_R(\tau,x) &= -\curl E_R(\tau,x),&\div B_R(\tau,x)&=0
\end{aligned}
\end{equation}
for all $(\tau,x)\in\R\times\R^3$, where we take as initial conditions that
\begin{equation}\label{eq:initial_condition}
E_R(\tau,x)=E^{(0)}(\tau,x)\text{ and }B_R(\tau,x)=B^{(0)}(\tau,x)\text{ for all }(\tau,x)\in(-\infty,0)\times\R^3.
\end{equation}
To ensure that these initial conditions are compatible with the equation system in \autoref{eq:Maxwell}, we silently assume that the fields $E^{(0)}(\tau,\cdot)$ and $B^{(0)}(\tau,\cdot)$ of the incident beam are for $\tau<0$ only supported outside the medium which shall be contained in the ball $\B_{r_{\mathrm s}}^3\subset\R^3$ around the origin with radius $r_{\mathrm s}>0$.

The interaction with the sample enters in this description via relations between the electric field $E_R$ and the displacement field $D_R$ as well as between the magnetic field $B_R$ and the magnetizing field $H_R$. If we assume that we are dealing with a non-magnetic medium, there will be no contribution of the sample to the magnetic fields and we get that $B_R=H_R$. Assuming further that it is a linear dielectric medium, we get that the induced polarization in the sample depends linearly on the strength of the electric field so that we have a relation of the form
\begin{equation}\label{eq:dielectric}
D_R(\tau,x) = E_R(\tau,x)+\int_0^\infty\chi_R(\tilde\tau,x)E_R(\tau-\tilde\tau,x)\d\tilde\tau.
\end{equation}
Here, the function $\chi_R\colon\R\times\R^3\to\R$ is the so-called electric susceptibility of the sample when it is rotated according to the rotation matrix $R^{-1}$. Defining $\chi\colon\R\times\R^3\to\R$ as the electric susceptibility of the medium in the reference state, this means that $\chi_R$ is given by $\chi_R(\tau,x)=\chi(\tau,Rx)$. Since the object shall be contained in $\B_{r_{\mathrm s}}^3$, we have that $\supp\chi(\tau,\cdot)\subset \B_{r_{\mathrm s}}^3$ for every $\tau\in[0,\infty)$.
And to be able to extend the integral in \autoref{eq:dielectric} to all of $\R$, we conveniently set $\chi(\tau,x)\coloneqq0$ for all $(\tau,x)\in(-\infty,0)\times\R^3$.

We want to perform a Fourier transform with respect to the time variable $\tau$ and denote by $\check g\in L^2(\R)$ the temporal Fourier transform of a function $g\in L^2(\R)$ following the convention
\[ \check g(\omega) \coloneqq \int_{-\infty}^\infty g(\tau)\e^{\i\omega\tau}\d\omega\text{ if }g\in C^\infty_{\mathrm c}(\R). \]
We thus switch to the temporal Fourier transform $\check E_R(\cdot,x)$ of the electric field $E_R(\cdot,x)$ at every position $x\in\R^3$ and combine Maxwell's macroscopic equations from \autoref{eq:Maxwell} with $B_R=H_R$ and \autoref{eq:dielectric} to the vector Helmholtz equations
\[ -\curl\curl\check E_R(\omega,x)+\frac{\omega^2}{c^2}(1+\check\chi_R(\omega,x))\check E_R(\omega,x) = 0\text{ for all }x\in\R^3 \]
for the functions $\check E_R(\omega,\cdot)$ for every frequency $\omega\in\R$.

By taking the divergence of these equations (or directly from $\div D_R=0$), we find the relation $\div\check E_R=-\div(\check\chi_R\check E_R)$ so that we can reduce each of them with the help of the vector identity $\curl\curl\check E_R=\grad\div\check E_R-\Delta\check E_R$ to a Helmholtz equation for each component of $\check E_R(\omega,\cdot)$:
\[ \Delta\check E_R(\omega,x)+\frac{\omega^2}{c^2}\check E_R(\omega,x) = -\left(\frac{\omega^2}{c^2}+\grad\div\right)(\check\chi_R\check E_R)(\omega,x)\text{ for all }\omega\in\R,\;x\in\R^3, \]
where we consider the right hand side as an unknown inhomogeneity. Since the Fourier transform $\check E^{(0)}$ is according to \autoref{eq:Maxwell_vacuum} a solution of the homogeneous equation, we obtain with the fundamental solution
\[ G(\tfrac\omega c,x)\coloneqq\frac{\e^{\i\frac\omega c\norm x}}{4\pi\norm x} \]
of the Helmholtz equation the integral equation
\[ \check E_R(\omega,x) = \check E^{(0)}(\omega,x)+\left(\frac{\omega^2}{c^2}+\grad\div\right)\int_{\R^3}G(\tfrac\omega c,x-y)\check\chi_R(\omega,y)\check E_R(\omega,y)\d y, \]
where the initial condition from \autoref{eq:initial_condition} is included by enforcing that $\omega\mapsto\check E_R(\omega,x)-\check E^{(0)}(\omega,x)$ can be holomorphically extended to a square integrable function on the upper half complex plane, which is ensured by the choice of the fundamental solution.

We consider a fixed frequency $\omega_0\in(0,\infty)$ for which the medium is only weakly scattering, meaning that $\abs{\check\chi_R(\omega_0,\cdot)}$ is sufficiently small. Under this assumption, we may approximate the solution by replacing $\check E_R(\omega_0,\cdot)$ in the integrand by $\check E^{(0)}(\omega_0,\cdot)$. The resulting approximate field $E_R^{(1)}$, given by
\[ \check E_R^{(1)}(\omega_0,x) \coloneqq \check E^{(0)}(\omega_0,x)+\left(\frac{\omega_0^2}{c^2}+\grad\div\right)\int_{\R^3}G(\tfrac{\omega_0}c,x-y)\check\chi_R(\omega_0,y)\check E^{(0)}(\omega_0,y)\d y, \]
is called the Born approximation.

We choose for the initial beam a plane wave with linear polarization along $e_1$ moving into the direction $e_3$, where $(e_j)_{j=1}^3$ denotes the standard basis in $\R^3$, that is, we write
\[ E^{(0)}(\tau,x) = u^{(0)}(\tau-\tfrac{x_3}c)e_1 \]
with some compactly supported function $u^{(0)}\colon\R\to\R$ which shall fulfill that $\supp u^{(0)}\subset(\frac{r_{\mathrm s}}c,\infty)$ (so that $E^{(0)}(\tau,x)=0$ for all $x\in \B_{r_{\mathrm s}}^3$ if $\tau<0$). Moreover, we want it to be essentially monochromatic with the frequency $\omega_0$, meaning that the support of $\check u^{(0)}$ should be localized around $\omega_0$ and $-\omega_0$.

The temporal Fourier transform of $E^{(0)}$ is given by $\check E^{(0)}(\omega,x) = \check u^{(0)}(\omega)\e^{\i\frac\omega cx_3}e_1$, and we thus arrive at the expression
\begin{equation}\label{eq:electric_field_Born}
\check E_{R,1}^{(1)}(\omega_0,x)-\check E_1^{(0)}(\omega_0,x) = \check u^{(0)}(\omega_0)\left(\frac{\omega_0^2}{c^2}+\partial_{x_1x_1}\right)\int_{\R^3}G(\tfrac{\omega_0}c,x-y)\check\chi_R(\omega_0,y)\e^{\i\frac{\omega_0}cy_3}\d y
\end{equation}
for the first component of the Born approximation of the electric field at the main frequency $\omega_0$ of the incident field. Similar expressions can be obtained for the other components, but we will assume that only the polarization in the direction of the incident field is measured and therefore only focus on this term.

To get measurements from different directions, we let the object rotate slowly and perform at different rotational states such an illumination. Idealized, this means that we obtain for a certain function $R\colon[0,T]\to \SO(3)$, which describes the rotation of the object, at every time step $t\in[0,T]$ a field component of the form $\check E_{R(t),1}^{(1)}(\omega_0,x)$.

We consider now two different measurement setups for recording this quantity.

\begin{description}
\item[Interferometric measurements:]
Using a polarization sensitive interferometric setup, it is indeed possible to obtain the values of the first component $\check E_{R(t),1}$ of the electric field on every point in a detector plane $\mathcal D\coloneqq\set{x\in\R^3}{x_3=r_{\mathrm m}}$ at some position $r_{\mathrm m}>r_{\mathrm s}$ outside the object. Assuming that the Born approximation at the considered frequency $\omega_0>0$ is close to the produced electric field, we can therefore acquire the diffraction tomography data
\[ m^{\mathrm{DT}}\colon[0,T]\times\R^2\to\C,\;m^{\mathrm{DT}}(t,x_1,x_2)\coloneqq(\check E_{R(t),1}^{(1)}-\check E_1^{(0)})(\omega_0,x_1,x_2,r_{\mathrm m}). \]

Since the formula in \autoref{eq:electric_field_Born} consists essentially of a convolution with $G$, we try to simplify it by performing a Fourier transform with respect to the variables $x_1$ and $x_2$. In the spatial domain, we use for the $d$-dimensional Fourier transform $\hat g$ of a function $g\in L^2(\R^d)$ the convention
\[ \hat g(k)\coloneqq(2\pi)^{-\frac{d}{2}}\int_{\R^d}g(x)\e^{-\i\inner{x}{k}}\d x\text{ if }g\in C_{\mathrm c}^\infty(\R^n;\C). \]

From our expression for $\check E_{R(t),1}^{(1)}$ in \autoref{eq:electric_field_Born}, we then get with the Weyl expansion for the fundamental solution $G$, which was derived in \cite{Wey19}, the relation
\begin{equation}\label{eq:dt_fourier-diffr}
\hat m^{\mathrm{DT}}(t,k) = \check u^{(0)}(k_0c)\left(k_0^2-k_1^2\right)\sqrt{\frac{\pi}{2}}\frac{\i \e^{\i\sqrt{k_0^2-\norm{k}^2}r_{\mathrm m}}}{\sqrt{k_0^2-\norm{k}^2}}\hat f^{\mathrm{DT}}\left(R(t) \begin{pmatrix} k\\h(\norm{k}) \end{pmatrix}\right)
\end{equation}
between the two-dimensional Fourier transform $\hat m^{\mathrm{DT}}(t,\cdot)$ of $m^{\mathrm{DT}}(t,\cdot)$ and the three-dimensional Fourier transform $\hat f^{\mathrm{DT}}$ of $f^{\mathrm{DT}}$, where we have set $k_0\coloneqq\frac{\omega_0}c$, defined the scattering potential $f^{\mathrm{DT}}(y)\coloneqq\check\chi(k_0c,y)$, and used the abbreviation
\begin{equation}\label{eq:dt_def_h}
h\colon[-k_0,k_0]\to[-k_0,0],\;h(\mu)\coloneqq\sqrt{k_0^2-\mu^2}-k_0.
\end{equation}
At least in the scalar case where the vector nature of the electric fields is neglected, this is known as the Fourier diffraction theorem and a derivation can be found in \cite[Section~5.2.4]{NatWub01}, for example.

\item[Intensity measurements:]
We also want to consider the case where only the intensities $\norm{\check E_{R(t)}}^2$ are measurable, as this is often more practical in experimental settings. In this case, the measured data is of the form
\begin{align*}
\norm{\check E_{R(t)}(\omega_0,x)}^2 &= \norm{\check E^{(0)}(\omega_0,x)}^2+2\Re\inner{\check E^{(0)}(\omega_0,x)}{\check E_{R(t)}(\omega_0,x)-\check E^{(0)}(\omega_0,x)} \\
&\qquad+\norm{\check E_{R(t)}(\omega_0,x)-\check E^{(0)}(\omega_0,x)}^2
\end{align*}
for all $x$ in the detector plane $\mathcal D$, where $\inner\cdot\cdot$ denotes the standard inner product in $\C^3$ with the convention that it is linear in the first and antilinear in the second argument and $\norm\cdot$ is the corresponding Euclidean norm in $\C^3$.

Following our assumption that the susceptibility $\abs{\check\chi(\omega_0,\cdot)}$ is small, we approximate $\check E_{R(t)}$ again by the Born approximation $\check E_{R(t)}^{(1)}$ and neglect the term $\norm{\check E_{R(t)}-\check E^{(0)}}^2$ as it is of second order in $\abs{\check\chi(\omega_0,\cdot)}$. Since the first term is known, we therefore effectively obtain, using again our expression form \autoref{eq:electric_field_Born}, the quantity
\begin{equation}\label{eq:raw_pb_measurements}
\begin{split}
&\Re\inner{\check E^{(0)}(\omega_0,x)}{(\check E_{R(t)}^{(1)}-\check E^{(0)})(\omega_0,x)} \\
&\qquad=\abs{\check u^{(0)}(\omega_0)}^2\left(\frac{\omega_0^2}{c^2}+\partial_{x_1x_1}\right)\Re\left(\e^{-\i k_0x_3}\int_{\R^3}G(\tfrac{\omega_0}c,x-y)\check\chi_R(\omega_0,y)\e^{\i\frac{\omega_0}cy_3}\d y\right)
\end{split}
\end{equation}
at all points $x\in\mathcal D$ and we can directly extract from this the real part of the integral.

To further simplify the problem in this case, we assume that we illuminate with a sufficiently high frequency $\omega_0>0$ so that we can approximate this data with its asymptotic limit as $\omega_0\to\infty$, see \cite[Section 3.3]{NatWub01}, for example.

We switch again to the wave number $k_0\coloneqq\frac{\omega_0}c$ and get, according to the stationary phase method, to be found in \cite[Theorem 7.7.6]{Hoe03}, for an arbitrary function $g\in C^\infty_{\mathrm c}(\R^3)$ with $\supp g\subset \B_{r_{\mathrm s}}^3$ for every value $y_3\in\R$ and every point $x\in\mathcal D$ in the limit $k_0\to\infty$ the asymptotic behavior
\begin{align*}
\int_{\R^2}G(k_0,x-y)g(y)\e^{\i k_0y_3}\d(y_1,y_2) &= \int_{\R^2}\frac{\e^{\i k_0(\norm{x-y}+y_3)}}{4\pi\norm{x-y}}g(y)\d(y_1,y_2) \\
&\simeq \frac{\i\e^{\i k_0(\abs{x_3-y_3}+y_3)}}{2k_0}g(x_1,x_2,y_3) = \frac{\i\e^{\i k_0r_{\mathrm m}}}{2k_0}g(x_1,x_2,y_3),
\end{align*}
since the only critical point of the function $\psi(y_1,y_2)\coloneqq\norm{x-y}+y_3$ is $(x_1,x_2)$ and the determinant of its Hessian matrix at that point is given by $\det(D^2\psi(x_1,x_2))^{-\frac12} = \abs{x_3-y_3}$.

Under the pretense that this is a good approximation for our expression in \autoref{eq:raw_pb_measurements}, we therefore consider
\begin{equation}\label{eq:pb_measurements}
m^{\mathrm{PB}}\colon[0,T]\times \R^2\to \R,\;m^{\mathrm{PB}}(t,x_1,x_2)\coloneqq\int_\R f^{\mathrm{PB}}(R(t)x) \d x_3.
\end{equation}
as the parallel-beam data for the imaging function $f^{\mathrm{PB}}\colon\R^3\to\R$, $f^{\mathrm{PB}}(y)\coloneqq\Im(\check\chi(k_0c,y))$. The two-dimensional spatial Fourier transform $\hat m^{\mathrm{PB}}(t,\cdot)$ of $m^{\mathrm{PB}}(t,\cdot)$ is then directly related to the three-dimensional Fourier transform $\hat f^{\mathrm{PB}}$ of $f^{\mathrm{PB}}$ via
\begin{equation}\label{eq:pb_measurements_Fourier}
\hat m^{\mathrm{PB}}(t,k) = \sqrt{2\pi}\hat f^{\mathrm{PB}}\left(R(t) \begin{pmatrix}k\\0\end{pmatrix}\right)\text{ for all }t\in[0,T],\;k\in\R^2,
\end{equation}
which is known as the Fourier slice theorem.

We remark that, although the physics behind electron microscopy is quite different, the measurements in single particle cryogenic electron microscopy have the same structure as we got in \autoref{eq:pb_measurements} with the only major difference that they are not recorded along a continuous rotational motion $R$, but only available at many discrete rotational states.
\end{description}

\subsection{Reconstruction problem}
Our aim is now to reconstruct the function $f^{\mathrm{DT}}$ or $f^{\mathrm{PB}}$, which describe the optical properties of the sample, from the measurement $m^{\mathrm{DT}}$ or $m^{\mathrm{PB}}$, respectively. If the rotational motion $R\colon[0,T]\to\SO(3)$ was known, this would be a classical inverse scattering problem and the Fourier coefficients of $f^{\mathrm{DT}}$ or $f^{\mathrm{PB}}$ could be directly read off from the measuremens via \autoref{eq:dt_fourier-diffr} and \autoref{eq:pb_measurements_Fourier} in these approximations.

However, the motion $R$ is not available in the experimental setup of \cite{LovPreThaRit21}, since the forces acting on the object depend on its internal structure which we want to detect with these measurements. The reconstruction problem thus takes a different turn in that we are now interested in recovering the unknown function $R$ from the measurement data, which would then allow us to reduce the problem back to the classical inverse scattering problem.

To explore the smooth motion $R$ of the object, we intend to show that it is generically possible to uniquely recover the change in rotation at every time step $t\in[0,T]$, that is, the derivative $R'(t)$, from the measurements in an infinitesimal neighborhood of $t$. Since $R(t)$ is in $\SO(3)$, we have that $R(t)^\top R'(t)$ is antisymmetric so that we can represent it with a three-dimensional vector.

\begin{definition}\label{def:ang_vel}
	Let $R\in C^1([0,T];\SO(3))$ be a time-dependent rotation. We define the corresponding angular velocity $\omega\colon [0,T]\to\R^3$ via the relation
	\begin{equation}\label{eq:def-omega}
		R(t)^\top R'(t) y=\omega(t)\times y \text{ for all }t\in[0,T],\;y\in\R^3.
	\end{equation}
\end{definition}

Given this angular velocity, the rotation can, of course, be determined by solving the linear ordinary differential equation in \autoref{eq:def-omega}, where we choose, considering the object at time $0$ as the reference state, as initial condition $R(0)$ to be the identity matrix $\I_{3\times3}$.

Since the direction $e_3$ is singled out as the illumination direction, it will be convenient to express the angular velocity in cylindrical coordinates, where we will conventionally use for a general vector $x\in\R^3$ the notation
	\begin{equation}\label{eq:def-cylind}
		x= \begin{pmatrix} \rho_x \phi_x \\ \zeta_x\end{pmatrix}=\begin{pmatrix} \rho_x \cos(\varphi_x)\\ \rho_x \sin(\varphi_x)\\ \zeta_x\end{pmatrix}
	\end{equation}
with the azimuth direction $\phi_x$ defined by the azimuth angle $\varphi_x\in [0,\pi)$, the cylindrical radius $\rho_x\in\R$, and the third component $\zeta_x\in\R$. Note that, in contrast to classical cylindrical coordinates, we allow negative radii $\rho_x$, but restrict in exchange the azimuth direction $\phi_x$ to the upper semicircle $S^1_+\coloneqq\set{(\cos(\alpha),\sin(\alpha))}{\alpha\in[0,\pi)}$.

To simplify the analysis, we are neglecting the possibility of translational motions of the object completely and we are conveniently assuming that the center of the functions $f^{\mathrm{DT}}$ and $f^{\mathrm{PB}}$, which we want to recover, is known to be in the origin so that the objects are only rotated around their center. Moreover, since it does not affect the theory, we consider both functions to be complex-valued, although at least $f^{\mathrm{PB}}$ is, according to the physical derivation, clearly real-valued. This leads us to a common underlying space of admissible objects for both reconstruction problems.

\begin{definition}\label{def:admis_obj}
	We define the set $\mathcal O$ of admissible objects as the set of complex-valued square integrable functions with vanishing first moments defined in the ball $\B_{r_{\mathrm s}}^3$ of some fixed radius $r_{\mathrm s}>0$ and center $0$:
	\begin{equation}\label{eq:obj}
		\mathcal{O}\coloneqq\set*{f\in L^2(\B^3_{r_{\mathrm s}})}{\int_{\B^3_{r_{\mathrm s}}}xf(x)\d x = 0}.
	\end{equation}
\end{definition}

\begin{lemma}\label{prop:admissible_obj_closed}
	The set $\mathcal{O}$ is a closed subspace of $L^2(\B^3_{r_{\mathrm s}})$.
\end{lemma}
\begin{proof}
	Let $(f_j)_{j=1}^\infty\subset\mathcal{O}$ be a converging sequence with limit $f\in L^2(\B_{r_{\mathrm s}}^3)$. Hölder's inequality then gives us
	\[ \norm*{\int_{\B^3_{r_{\mathrm s}}}x f(x) \d x} = \norm*{\int_{\B^3_{r_{\mathrm s}}}x (f(x)-f_j(x)) \d x} \le \left(\int_{\B^3_{r_{\mathrm s}}}\norm{x}^2\d x\right)^{\frac{1}{2}}\norm{f-f_j}_{L^2}, \]
	where right hand side converges to $0$ as $j\to\infty$ so that $f\in\mathcal{O}$.
\end{proof}

\section{Reconstruction for diffraction tomography data}\label{section:dt}
We will first discuss the reconstruction problem of the rotational motion from the diffraction tomography data given in \autoref{eq:dt_fourier-diffr}. To simplify the notation, let us get rid of the unneccessary constants.

\begin{definition}\label{def:pb_measurements}
	Let $f\in\mathcal{O}$ be an admissible object and $R\colon[0,T]\to\SO(3)$ be a rotational motion. We define the diffraction tomography measurements $\hat m$ of $f$ under the rotation $R$ as
		\begin{equation}\label{eq:dt_def_meas}
				\hat m\colon[0,T]\times\B_{k_0}^2\to\C,\;\hat m(t,k)\coloneqq \hat f\left(R(t) \begin{pmatrix}k\\h(\norm{k})\end{pmatrix}\right),
		\end{equation}
where $h$ is given by \autoref{eq:dt_def_h}.
\end{definition}

The main question we want to answer now is the following.
\begin{problem}\label{pr:dt}
Let $T>0$ and $k_0>0$ be some fixed parameters. Under which conditions on the function $f\in\mathcal O$, is it possible to uniquely recover a rotation $R\in C^1([0,T];\SO(3))$ with the normalization $R(0)=\I_{3\times3}$ from the corresponding diffraction tomography measurements $\hat m\colon[0,T]\times\B_{k_0}^2\to\C$?
\end{problem}

Geometrically, the measurements $\hat m(t,\cdot)$ at a given time $t$ are given by function values of $\hat f$ on a hemisphere that is rotated according to $R(t)$. When considering measurements at two distinct times $s, t \in [0, T]$, the corresponding rotated hemispheres typically intersect in a circle or a circular arc. Parametrizing a point on this intersection by $R(t)(k,h(\norm k))=R(s)(\tilde{k},h(\norm{\tilde{k}})$ with two values $k, \tilde{k} \in \B_{k_0}^2$, we see that
\[ \hat m(t,k) = \hat{f}\left(R(t) \begin{pmatrix} k \\ h(\|k\|) \end{pmatrix}\right) = \hat{f}\left(R(s) \begin{pmatrix} \tilde{k} \\ h(\|\tilde{k}\|) \end{pmatrix}\right) = \hat m(s,\tilde k). \]  
This relationship is referred to as the common circle equation, which gives us, by identifying in the two data sets $m(t,\cdot)$ and $m(s,\cdot)$ pairs $(k,\tilde k)\in\B_{k_0}^2\times B_{k_0}^2$ of points with the same measurement values, a condition on the relative orientation $R(s)^\top R(t)$ of the object between the time steps $s$ and $t$. The approach to derive from this the motion $R$ is commonly known as the common circle or common arc method and we refer to the articles~\cite{HulSzoHaj03,BorTeg11,QueElbSchSte24} for further details.

In \cite[Section 4]{QueElbSchSte24}, an infinitesimal version of the common circle equation was derived by taking for a differentiable motion $R$ the limit $s\to t$.
Using the angular velocity $\omega$ associated to $R$, written in cylindrical coordinates as specified in \autoref{eq:def-cylind}, and setting $k^\perp\coloneqq(-k_2,k_1)$ for an arbitrary vector $k=(k_1,k_2)\in\R^2$, this infinitesimal common circle equation for the values $\rho_\omega(t)$, $\phi_\omega(t)$, and $\zeta_\omega(t)$ reads, according to \cite[Lemma 4.1]{QueElbSchSte24}, at every time step $t\in[0,T]$
\begin{equation}\label{eq:dt_inf-common-circle}	
	\partial_t\hat m(t,\mu \phi_\omega(t))=\big(\mu\zeta_\omega(t)-\rho_\omega(t)h(\mu)\big)\inner{\nabla_k\hat m(t,\mu\phi_\omega(t))}{\phi_\omega(t)^\perp}\text{ for all }\mu\in(-k_0,k_0).
\end{equation}
Here, we denote by $\inner\cdot\cdot$ the standard inner product on $\C^3$ which is linear in the first argument, and the gradient is taken with respect to the variable $k$ only.

\subsection{Characterization of objects which allow for a unique rotation reconstruction}

A priori, it is not clear at all if \autoref{eq:dt_inf-common-circle} is a sufficient condition to uniquely recover the angular velocity $\omega(t)$ at the time $t$. And in the extremal case of a spherically symmetric function $f$, both sides of the equation are in fact trivially zero for every choice of values $\rho_\omega(t)$, $\phi_\omega(t)$, and $\zeta_\omega(t)$.

However, if the function $f$ is sufficiently irregular, it turns out that we can indeed guarantee that \autoref{eq:dt_inf-common-circle} possesses a unique solution.

\begin{definition}\label{def:dt_symmetry}
	We call a function $f\in\mathcal{O}$ DT-symmetric if there exist unit vectors $\xi,\eta,\nu\in S^2$, where $S^2\coloneqq\set{x\in\R^3}{\norm{x}=1}$, with $\inner{\xi}{\eta}=0$ and such that we have 
		\begin{equation}\label{eq:dt_symmetric}
			\inner{\nabla\hat f(\mu\xi+h(\mu)\eta)}{(\mu\xi+h(\mu)\eta)\times\nu}=0 \text{ for all } \mu\in(-k_0,k_0),
		\end{equation}
	where $h$ is given by \autoref{eq:dt_def_h}; otherwise we call it DT-asymmetric.
	Moreover, we denote the set of admissible DT-symmetric functions by $\mathcal{S}_{\mathrm{DT}}$.
\end{definition}

An object is hence DT-symmetric if the gradient of its Fourier transform evaluated along a semicircle given by $\mu\mapsto\mu\xi+h(\mu)\eta$ only varies in certain planes depending on $\mu$. In the special case $\nu\in\spn\{\xi,\eta\}$, we have that the vectors $(\mu\xi+h(\mu)\eta)\times\nu$ and $\xi\times\eta$ are for every $\mu\in(-k_0,k_0)$ parallel so that 
\autoref{eq:dt_symmetric} then reads
	\begin{equation}\label{eq:dt_symmetric_reduced}
		\inner{\nabla\hat f(\mu\xi+h(\mu)\eta)}{\xi\times\eta}=0 \text{ for all } \mu\in(-k_0,k_0).
	\end{equation}

While the definition of DT-symmetry might not be intuitively interpretable in the spatial domain, it surely includes mirror symmetric objects. 

\begin{example}\label{ex:dt_mirror_sym}
	Consider for arbitrary vectors $\xi,\eta\in S^2$ with $\inner{\xi}{\eta}=0$ a mirror symmetric object $f\in\mathcal O$ with the mirror plane $M\coloneqq\set{x\in\R^3}{\inner{\xi\times\eta}{x}=0}$, that is, we have that $f(x_1\xi+x_2\eta+x_3\xi\times\eta)=f(x_1\xi+x_2\eta-x_3\xi\times\eta)$ for all $x_1,x_2,x_3\in\R$. Its Fourier transform $\hat f$ inherits this mirror symmetry of the original function $f$ so that the gradient $\nabla\hat f(\kappa)$ must therefore lie at every point $\kappa\in M$ in the mirror plane $M$.
The object thus fulfills the DT-symmetry condition
given in \autoref{eq:dt_symmetric_reduced}.
\end{example}

This DT-asymmetry of an object now proves to be exactly the condition which is needed for the angular velocity $\omega$ corresponding to the true rotation $R$ to be the unique solution of the infinitesimal common circle equation at every time step.
In particular, the following proposition contains a proof of \autoref{eq:dt_inf-common-circle}, similar to the one from \cite[Lemma~4.1]{QueElbSchSte24}, as we demonstrate that \autoref{eq:dt_inf-common-circle} is satisfied by the angular velocity.

\begin{proposition}\label{thm:dt_uniqueness}
	Let $\hat m$ be the diffraction tomography measurements of a DT-asymmetric function $f\in\mathcal{O}\setminus\mathcal{S}_{\mathrm{DT}}$ under the rotation $R\in C^1([0,T];\SO(3))$ with associated angular velocity $\omega$ written in cylindrical coordinates $\rho_\omega, \phi_\omega$, and $\zeta_\omega$ as in \autoref{eq:def-cylind}.

A vector $u\in\R^3$, written in cylindrical coordinates $\rho_u, \phi_u$ and $\zeta_u$, then solves the infinitesimal common circle equation
	\begin{equation}\label{eq:dt_common_circle_unique}
		\partial_t\hat m(t,\mu \phi_u)=\left(\mu\zeta_u-h(\mu)\rho_u\right)\inner{\nabla_k\hat	m(t,\mu\phi_u)}{\phi_u^\perp} \text{ for every } \mu\in(-k_0,k_0)
	\end{equation}
at a time $t\in[0,T]$ if and only if $u=\omega(t)$.
\end{proposition}
\begin{proof}
	We use \autoref{def:pb_measurements} of the measurements $\hat m$ together with \autoref{def:ang_vel} of the angular velocity and compute the derivatives to express \autoref{eq:dt_common_circle_unique} in terms of the functions $\hat f$ and $R$:
	\[\inner*{\nabla \hat f\left(R(t) \begin{pmatrix}
			\mu\phi_u \\h(\mu)
		\end{pmatrix}\right)}{R(t)\left(\omega(t)\times \begin{pmatrix}
			\mu\phi_u \\h(\mu)
		\end{pmatrix}\right)} = (\mu\zeta_u-h(\mu)\rho_u)\inner*{\nabla \hat f\left(R(t) \begin{pmatrix}
			\mu\phi_u \\h(\mu)
		\end{pmatrix}\right)}{R(t) \begin{pmatrix}
		\phi_u^\perp \\0
	\end{pmatrix}}.\]
	Expanding the vector product, 
	\begin{align*}
		\omega(t)\times \begin{pmatrix}
			\mu\phi_u \\h(\mu)
		\end{pmatrix} &= \begin{pmatrix}\rho_\omega(t)\phi_\omega(t)\\0\end{pmatrix}\times\begin{pmatrix}\mu\phi_u\\0\end{pmatrix}+\begin{pmatrix}\rho_\omega(t)\phi_\omega(t)\\0\end{pmatrix}\times\begin{pmatrix}0\\h(\mu)\end{pmatrix}+\begin{pmatrix}0\\\zeta_\omega(t)\end{pmatrix}\times\begin{pmatrix}\mu\phi_u\\0\end{pmatrix} \\
		&= \mu\rho_\omega(t)\begin{pmatrix}0\\\inner{\phi^\perp_\omega(t)}{\phi_u}\end{pmatrix}-h(\mu)\rho_\omega(t)\begin{pmatrix}\phi^\perp_\omega(t)\\0\end{pmatrix}+\mu\zeta_\omega(t)\begin{pmatrix}\phi_u^\perp\\0\end{pmatrix},
	\end{align*}
	we see that this is equivalent to the relation
		\begin{equation}\label{eq:dt_common_circle_unique_simpl}
			\inner*{\nabla \hat f\left(R(t) \begin{pmatrix}
					\mu\phi_u \\h(\mu)
				\end{pmatrix}\right)}{R(t)\left(\mu\begin{pmatrix}(\zeta_\omega(t)-\zeta_u)\phi_u^\perp\\\rho_\omega(t)\inner*{\phi^\perp_\omega(t)}{\phi_u}\end{pmatrix}
				+h(\mu)\begin{pmatrix}\rho_u\phi_u^\perp-\rho_\omega(t)\phi^\perp_\omega(t)\\0\end{pmatrix}\right)}=0.
		\end{equation}
	From this, it is clear that the choice $\zeta_u=\zeta_\omega(t)$ and $\rho_u\phi_u=\rho_\omega(t)\phi_\omega(t)$ is a solution. 
	
	It remains to show that this solution is indeed unique.
	So, let $(\rho_u,\phi_u,\zeta_u)$ be a solution and define the orthonormal vectors
	\[ \xi\coloneqq R(t)\begin{pmatrix}\phi_u\\0\end{pmatrix}\text{ and }\eta\coloneqq R(t)\begin{pmatrix}
		0\\ 1
	\end{pmatrix}, \]
	so that we have
	\[ R(t)\begin{pmatrix}
		\mu\phi_u \\h(\mu)
	\end{pmatrix} = \mu \xi+h(\mu)\eta\text{ and }\xi\times\eta=-R(t)\begin{pmatrix}\phi^\perp_u\\0\end{pmatrix}. \]
	
	Observing that 
	\begin{align*}
&\left(R(t)(\omega(t)-u)\right)\times\xi=R(t)\left(\begin{pmatrix}
		\rho_\omega(t) \phi_\omega(t)-\rho_u \phi_u\\ \zeta_\omega(t)-\zeta_u
	\end{pmatrix}\times\begin{pmatrix}
		\phi_u\\ 0
	\end{pmatrix} \right)=R(t)\begin{pmatrix}(\zeta_\omega(t)-\zeta_u)\phi_u^\perp\\\rho_\omega(t)\inner{\phi^\perp_\omega(t)}{\phi_u}\end{pmatrix} \text{ and} \\
&\left(R(t)(\omega(t)-u)\right)\times\eta=R(t)\left(\begin{pmatrix}
		\rho_\omega(t) \phi_\omega(t)-\rho_u \phi_u\\ \zeta_\omega(t)-\zeta_u
	\end{pmatrix}\times\begin{pmatrix}
		0\\ 1
	\end{pmatrix} \right)=R(t)\begin{pmatrix}\rho_u\phi_u^\perp-\rho_\omega(t)\phi^\perp_\omega(t)\\0\end{pmatrix},
\end{align*}
	we find that \autoref{eq:dt_common_circle_unique_simpl} can be written equivalently as
		\begin{equation}\label{eq:dt_common_circle_unique_simpl2}
			\inner{\nabla \hat f\left(\mu \xi+h(\mu)\eta\right)}{\left(R(t)(\omega(t)-u)\right)\times(\mu \xi+h(\mu)\eta)}=0\text{ for all }\mu\in(-k_0,k_0).
		\end{equation}
	If we had $u\neq\omega(t)$ now, then this would imply that $f$ satisfies the DT-symmetry condition from \autoref{eq:dt_symmetric} with a unit vector $\nu$ parallel to $R(t)(\omega(t) - u)$. Since $f\notin\mathcal S_{\mathrm{DT}}$, we thus necessarily have $u=\omega(t)$.
\end{proof}

We quickly want to take a look how this non-uniqueness looks like for mirror-symmetric objects.

\begin{example}
Let $R\colon[0,T]\to\SO(3)$ be a rotational motion and $\omega$ be the associated angular velocity. Assume that the object $f\in\mathcal O$ has a mirror symmetry and that the corresponding mirror plane for the rotated object $x\mapsto f(R(t)x)$
is at some time $t\in[0,T]$ perpendicular to the vector $(\phi_\omega^\perp(t),0)$, where we represent $\omega$ in cylindrical coordinates as before. With the unit vectors
\[ \xi=R(t)\begin{pmatrix} \phi_\omega(t)\\0 \end{pmatrix}\text{ and }\eta=R(t)\begin{pmatrix} 0\\1 \end{pmatrix}, \]
this means that $f(x_1\xi+x_2\eta+x_3\xi\times\eta) = f(x_1\xi+x_2\eta-x_3\xi\times\eta)$ for all $x_1,x_2,x_3\in\R$.

Since we then have that $\nabla\hat f(\kappa)$ is orthogonal to $\xi\times\eta$ for every $\kappa\in\spn\{\xi,\eta\}$, the infinitesimal common circle equation in the form of \autoref{eq:dt_common_circle_unique_simpl2} is fulfilled whenever $R(t)(\omega(t)-u) \in \spn\{\xi,\eta\}$, that is, for all
\[ u \in \set*{\begin{pmatrix}\rho\phi_\omega(t)\\\zeta\end{pmatrix}}{\rho,\zeta\in\R}. \]
\end{example}

\subsection{DT-asymmetric point sets and construction of DT-asymmetric functions}\label{section:DT_asym_construction}

To get a better understanding what sort of functions are DT-asymmetric and thus unproblematic to reconstruct, we restrict our attention for the moment to objects which essentially consist of a sum of finitely many point-like particles. The DT-asymmetry condition from \autoref{def:dt_symmetry} then translates into a condition for the positions of these point masses.

To formulate it, let us denote for every $\xi\in S^2$ by
	\begin{equation}\label{eq:det_orth_proj}
		\pi_\xi\colon\R^3\to\R^3,\;\pi_\xi(x)\coloneqq x-\inner{x}{\xi}\xi,
	\end{equation}
the orthogonal projection onto the hyperplane orthogonal to $\xi$.

\begin{definition}\label{def:dt_asym_points}
We call a finite set $P\subset\R^3$ of points a DT-asymmetric point set if we can select for every direction $\xi\in S^2$ two points $p_1(\xi),p_2(\xi)\in P$ such that we have for every $j\in\{1,2\}$ that
\begin{enumerate}
\item \label{cond:DT_asym_points_cond1}
there does not exist any point $p\in P\setminus\{p_j(\xi)\}$ and constant $c\in \R$ for which
\begin{equation}\label{eq:dt_asym_points_cond1}
\pi_\xi(p_j(\xi))= c \pi_\xi(p),
\end{equation}
that is, the orthogonal projections $\pi_\xi(p_j(\xi))$ and $\pi_\xi(p)$ of $p_j(\xi)$ and every other point $p\in P\setminus\{p_j(\xi)\}$ are not parallel unless $\pi_\xi(p)=0$, and

\item \label{cond:DT_asym_points_cond2}
the projection of $p_j(\xi)$ onto $\xi$ is nonzero:
\begin{equation}\label{eq:dt_asym_points_cond2}
\inner{p_j(\xi)}\xi\neq 0.
\end{equation}
\end{enumerate}
\end{definition}

In order for the condition in \autoref{cond:DT_asym_points_cond2} of \autoref{def:dt_asym_points} to be fulfilled, we have to find for every direction $\xi\in S^2$ two points in the set $P$ that do not lie in the linear subspace orthogonal to $\xi$. This is guaranteed if $P$ consists of at least four points of which any three distinct points are linearly independent.

However, the condition in \autoref{cond:DT_asym_points_cond1} will never be fulfilled for such a set $P=\{p_j\}_{j=1}^4$ with only four points, since we can then consider the two planes $\spn\{p_1,p_2\}$ and $\spn\{p_3,p_4\}$ and, for a vector $\xi\in S^2$ in the intersection of these planes, the orthogonal projections $\pi_\xi(p_1)$ and $\pi_\xi(p_2)$ of $p_1$ and $p_2$ as well as the orthogonal projections $\pi_\xi(p_3)$ and $\pi_\xi(p_4)$ of $p_3$ and $p_4$ will both be parallel to each other.

Thus, a set $P\supset\{p_j\}_{j=1}^6$ with at least six points is required to introduce a third plane $\spn\{p_5, p_6\}$ that does not intersect the original two planes $\spn\{p_1, p_2\}$ and $\spn\{p_3, p_4\}$ in a common line.

Since the two points fulfilling the first condition are not necessarily those for which the second condition holds, we will take a set with at least eight points such that we always have a set of four points of which every pair will satisfy the condition in \autoref{cond:DT_asym_points_cond1} of \autoref{def:dt_asym_points} and from which we can then choose a pair satisfying simultaneously the condition in \autoref{cond:DT_asym_points_cond2}.

\begin{lemma}\label{thm:dt_asym_points_existence}
	Let $N\geq8$ and $P=\{p_j\}_{j=1}^N\subset\R^3\setminus\{0\}$ be a finite set of points with the following properties:
		\begin{enumerate}
			\item $\det(p_i,p_j,p_k)\neq 0\text{ for all distinct }i,j,k\in\{1,\dots,N\}$ and \label{eq:DT_point_exist_cond1}
			\item $\det(p_i\times p_j,p_k\times p_\ell,p_m\times p_n)\neq 0\text{ for all distinct }i,j,k,\ell,m,n\in\{1,\dots,N\}$. \label{eq:DT_point_exist_cond2}
		\end{enumerate}
	Then, $P$ is a DT-asymmetric point set. 
\end{lemma}

\begin{proof}
We consider the set $\mathcal E\coloneqq\set{\spn\{p_i,p_j\}}{i,j\in\{1,\ldots,N\},\,i\ne j}$ of all linear subspaces spanned by two points in $P$. The condition in \autoref{eq:DT_point_exist_cond1} then tells us that $E\cap P$ consists for every plane $E\in\mathcal E$ of exactly two points. And the property in \autoref{eq:DT_point_exist_cond2} states that the normal vectors of three different planes $E_1,E_2,E_3\in\mathcal E$ are linearly independent so that the subspaces $E_1$, $E_2$, and $E_3$ do not intersect in a common line.

Let now $\xi\in S^2$ be an arbitrarily given vector.
\begin{itemize}
\item
If $\xi$ is not parallel to any point in $P$, the set $\mathcal E_0\subset\mathcal E$ of all planes $E\in\mathcal E$ with $\xi\in E$ can therefore have at most two elements. The set $P_0\coloneqq\bigcup_{E\in\mathcal E_0}E\cap P$ of all points spanning the planes in $\mathcal E_0$ thus consists of at most four points. And for all other points $p\in P\setminus P_0$, we have that $\spn\{\xi,p\}\cap P=\{p\}$, meaning that the projection $\pi_\xi(p)$, defined as in \autoref{eq:det_orth_proj}, is not parallel to any vector $\pi_\xi(\tilde p)$ for $\tilde p\in P\setminus\{p\}$.

Out of the at least four points in $P\setminus P_0$, there can be at most two orthogonal to $\xi$ so that we can find two points $p_1(\xi),p_2(\xi)\in P\setminus P_0$ which satisfy \autoref{eq:dt_asym_points_cond2}.
\item
If $\xi$ is parallel to some element $q\in P$, then we have for every $p\in P$ that $\spn\{\xi,p\}=\spn\{q,p\}\in\mathcal E$ so that $\spn\{\xi,p\}\cap P=\{q,p\}$, which means that $\pi_\xi(p)$ is not parallel to $\pi_\xi(\tilde p)$ for any $\tilde p\in P\setminus\{q,p\}$ and $\pi_\xi(q)=0$.

We therefore choose arbitrary points $p_1(\xi),p_2(\xi)\in P\setminus\{q\}$ which are not orthogonal to $\xi$.
\end{itemize}

In both cases, the points $p_1(\xi),p_2(\xi)\in P$ then fulfill by construction the conditions in \autoref{def:dt_asym_points}.
\end{proof}

\begin{remark}\label{rem:dt_asym_points_construction}
This shows that we can iteratively construct a DT-asymmetric set with arbitrarily many elements $N\ge8$ by starting at $P_0\coloneqq\emptyset$ and successively adding points $p_{n+1}\in\R^3\setminus\{0\}$, $n\in\{1,\ldots,N-1\}$, such that $P_{n+1}\coloneqq P_n\cup\{p_{n+1}\}$ fulfills the properties in \autoref{eq:DT_point_exist_cond1} and \autoref{eq:DT_point_exist_cond2} of \autoref{thm:dt_asym_points_existence}. Since these conditions would only fail if $p_{n+1}$ is chosen on a union of finitely many certain two-dimensional subspaces ($p_{n+1}$ should not lie on a subspace spanned by two points of $P_n$ and it should not lie on a subspace spanned by the intersection of two such subspaces and another point in $P_n$), we can even choose $p_{n+1}$ to lie in a predefined relatively open set on a sphere around the origin.
\end{remark}

To ensure that the DT-asymmetric function we construct will belong to the set $\mathcal{O}$ of admissible objects, which means, in particular, that it has to have vanishing first-order moments, we determine nonzero weights such that the weighted sum of all points in the underlying DT-asymmetric point set is in the origin.

\begin{lemma}\label{thm:dt_weights}
	Let $N\geq4$ and $P=\{p_j\}_{j=1}^N\subset\R^3\setminus\{0\}$ be a finite set of points satisfying
		\[\det(p_i,p_j,p_k)\neq 0\text{ for all distinct }i,j,k\in\{1,\dots,N\}.\]
	Then, there exists a corresponding set $\{w_j\}_{j=1}^N\subset\R\setminus\{0\}$ of weights such that $\sum_{j=1}^{N} w_j p_j=0$.
\end{lemma}
\begin{proof}
	We proceed by induction on $N$.

	Since any three points in $P$ are linearly independent, we can express $p_4$ uniquely as $p_4=\sum_{j=1}^3 a_j p_j$, where each coefficient $a_j \neq 0$. Thus, setting $w_j \coloneqq a_j$ for $j=1,2,3$ and $w_4 \coloneqq -1$, guarantees the result for the base case $N=4$.

	If the statement holds for $N-1$ for some $N\ge5$, there exists for the subset $\{p_j\}_{j=1}^{N-1}$ of $P$ a set $\{w_j\}_{j=1}^{N-1}\subset\R\setminus\{0\}$ with $\sum_{j=1}^{N-1}w_j p_j=0$.
	By the linear independence of any three points, we find nonzero coefficients $a_{N-3}, a_{N-2}, a_{N-1} \in \R\setminus \{0\}$ such that
	\[
	p_N = a_{N-3} p_{N-3} + a_{N-2} p_{N-2} + a_{N-1} p_{N-1}.
	\]
	
	Multiplying this by an arbitrary $\lambda\in\R\setminus\{0\}$ and adding it to our induction hypothesis yields
	\[
	\sum_{j=1}^{N-4} w_j p_j + (w_{N-3} - \lambda a_{N-3}) p_{N-3} + (w_{N-2} - \lambda a_{N-2}) p_{N-2} + (w_{N-1} - \lambda a_{N-1}) p_{N-1} + \lambda p_N = 0.
	\]
	
	Since all the $w_j$ and $a_j$ are nonzero, we can now choose $\lambda\in\R\setminus\{0\}$ such that none of the coefficients of the $p_j$, $j\in\{1,\ldots,N\}$, vanish, proving the statement for $N$.
\end{proof}

By placing radially symmetric functions at the points of a DT-asymmetric set and scaling them correspondingly to ensure that the first-order moments vanish, we now get a DT-asymmetric function.

\begin{lemma}\label{thm:dt_def_dist}
	Let $P=\{p_j\}_{j=1}^N\subset \B_{r_{\mathrm s}}^3$ be a DT-asymmetric point set and $\{w_j\}_{j=1}^N\subset\R\setminus\{0\}$ be corresponding weights with $\sum_{j=1}^{N} w_j p_j=0$.
	Moreover, let $\psi\in L^2(\B_{r_{\mathrm s}}^3)\setminus\{0\}$ be a radial function with $\supp(\psi)\subset \B^3_\delta$, where $\delta\coloneqq\mathrm{dist}(P,\partial \B_{r_{\mathrm s}}^3)=\min_{p\in P}\abs{\norm{p}-r_{\mathrm s}}$.

	The function
	\begin{equation}\label{eq:dt_def_dist}
		\Psi\colon\B_{r_{\mathrm s}}^3\to\C,\;\Psi(x)\coloneqq\sum_{j=1}^{N}w_j \psi(x-p_j),
	\end{equation}
	in $\mathcal O$ is then DT-asymmetric.
\end{lemma}
\begin{proof}
	We first note that $\Psi$ is indeed contained in the set of admissible functions $\mathcal{O}$ as introduced in \autoref{def:admis_obj}, since we have
	\[\int_{\R^3}x\Psi(x)\d x = \left(\int_{\R^3}x\psi(x)\d x\right)\sum_{j=1}^{N} w_j+\left(\int_{\R^3}\psi(x)\d x\right) \sum_{j=1}^{N} w_j p_j=0 \]
	due to $\psi$ being radial and due to the definition of the weights $w_j$.

	The Fourier transform $\hat\Psi$ of $\Psi$ and its gradient can be calculated explicitly and we find
	\begin{align}
			&\hat\Psi(k)=\hat\psi(k)\sum_{j=1}^N w_j \e^{-\i\inner{p_j}{k}}\text{ and}\label{eq:dt_psi_fourier}\\
			&\nabla\hat\Psi(k)=\nabla\hat\psi(k)\sum_{j=1}^N w_j \e^{-\i\inner{p_j}{k}}-\i\hat\psi(k)\sum_{j=1}^N w_j \e^{-\i\inner{p_j}{k}}p_j.\label{eq:dt_psi_fourier_grad}
	\end{align}
	
	Let us assume by contradiction that $\Psi$ was DT-symmetric so that we could find $\xi$,$\eta,\nu\in S^2$ with $\inner\xi\eta=0$ such that
	\begin{equation}\label{eq:dt_dist_cond}
		\inner{\nabla \hat \Psi(\mu\xi+h(\mu)\eta)}{(\mu\xi+h(\mu)\eta)\times\nu}=0 \text{ for all }\mu\in(-k_0,k_0).
	\end{equation}
	Since $\psi$ is radial, which implies that $\nabla\hat\psi(\mu\xi + h(\mu)\eta)$ is parallel to $\mu\xi + h(\mu)\eta$, we have that
		\[\inner{\nabla\hat\psi(\mu\xi + h(\mu)\eta)}{(\mu\xi + h(\mu)\eta) \times \nu}=0.\]
	Substituting the \autoref{eq:dt_psi_fourier_grad} for the gradient of $\hat\Psi$ into \autoref{eq:dt_dist_cond} would hence yield
		\[	\hat\psi(\mu\xi + h(\mu)\eta) \sum_{j=1}^N w_j \e^{-\i\inner{p_j}{\mu\xi + h(\mu)\eta}} \inner{p_j}{(\mu\xi + h(\mu)\eta) \times \nu} = 0. 	\]
	Furthermore, since $\psi$ is radial and has compact support, its Fourier transform $\hat\psi$ is radial and analytic. Consequently, there exists a nonempty interval $I \subset (-k_0, k_0)$ such that $\hat\psi(\mu\xi + h(\mu)\eta) \neq 0$ for all $\mu \in I$. (Otherwise, we would have an accumulation point of zeros of $\mu\mapsto\hat\psi(\mu\xi+h(\mu)\eta)$ so that the function would vanish by analyticity on the whole interval $(-k_0,k_0)$, which means that the radial, analytic function $\hat\psi$ would vanish in some nonempty open spherical shell and therefore everywhere.) Restricting the equation to this interval $I$, we would thus have that
		\begin{equation}\label{eq:dt_dist_cond_red}
			\sum_{j=1}^{N}w_j\inner{p_j}{\xi\times\nu}\mu  q_j(\mu)+\sum_{j=1}^{N}w_j\inner{p_j}{\eta\times\nu}h(\mu) q_j(\mu)=0\text{ for all }\mu\in I,
		\end{equation}
	where we defined $q_j(\mu)\coloneqq\e^{-\i(\inner{p_j}\xi\mu+\inner{p_j}\eta h(\mu))}$.

	Based on the DT-asymmetry property of $P$, we now choose two points $p_1(\xi\times\eta),p_2(\xi\times\eta)\in P$ satisfying the conditions given in \autoref{cond:DT_asym_points_cond1} and \autoref{cond:DT_asym_points_cond2} in \autoref{def:dt_asym_points} with respect to the direction $\xi\times\eta$, that is, we impose for each $j\in\{1,2\}$ that there does not exist any point $p\in P\setminus\{p_j(\xi\times\eta)\}$ and value $c\in\R$ with
		\begin{equation}\label{cond:DT_asym_points_cond1_fixed_dir}
			\pi_{\xi\times\eta}(p_j(\xi\times\eta))= c \pi_{\xi\times\eta}(p),
		\end{equation}
	where we define the projection $\pi_{\xi\times\eta}$ as in \autoref{eq:det_orth_proj}, and that $p_j(\xi\times\eta)$ is not orthogonal to $\xi\times\eta$:
		\begin{equation}\label{cond:DT_asym_points_cond2_fixed_dir}
			\inner{p_j(\xi\times\eta)}{\xi\times\eta}\neq 0.
		\end{equation}
To simplify the notation, let us assume that the elements in $P$ are numbered such that $p_1(\xi\times\eta)=p_1$ and $p_2(\xi\times\eta)=p_2$.

Since the orthogonal projections of $p_1$ and $p_2$ with respect to $\xi\times\eta$ are then by the first condition not parallel to any other nonzero orthogonal projection of a point from $P$, neither of the two vectors $(\inner{p_1}{\xi},\inner{p_1}{\eta})$ and $(\inner{p_2}{\xi},\inner{p_2}{\eta})$ is equal to any other vector $(\inner{p}{\xi},\inner p{\eta})$, $p\in P\setminus\{p_1,p_2\}$.
From \autoref{thm:dt_compl_exp_linindep}, we thus get that the five functions
\[ \big(\mu\mapsto \mu q_j(\mu)\big)_{j=1}^2,\;\big(\mu\mapsto h(\mu)q_j(\mu)\big)_{j=1}^2,\text{ and }\mu\mapsto\sum_{j=3}^N w_j\big(\inner{p_j}{\xi\times\nu}\mu+\inner{p_j}{\eta\times\nu}h(\mu)\big)q_j(\mu) \]
are linearly independent on $I$. \autoref{eq:dt_dist_cond_red} can therefore only hold if
		\begin{equation}\label{eq:dist_coeff_vanish}
			\inner{p_1}{\xi\times\nu}=\inner{p_2}{\xi\times\nu}=\inner{p_1}{\eta\times\nu}=\inner{p_2}{\eta\times\nu}=0.
		\end{equation}
	\begin{itemize}
		\item If $\nu\notin\spn\{\xi,\eta\}$, the vectors $\xi,\eta,\nu\in S^2$ are linearly independent and \autoref{eq:dist_coeff_vanish} would imply that $p_1$ and $p_2$ are parallel to $\nu$. But then also $\pi_{\xi\times\eta}(p_1)$ and $\pi_{\xi\times\eta}(p_2)$ would be parallel, contradicting the condition that there should be no $p\in P\setminus\{p_j\}$ fulfilling \autoref{cond:DT_asym_points_cond1_fixed_dir} for $j\in\{1,2\}$.

		\item If $\nu\in\spn\{\xi,\eta\}$, we can write it as $\nu=\alpha\xi+\beta\eta\in S^2$ for some $\alpha,\beta\in\R$ so that $\xi\times\nu= \beta (\xi\times\eta)$ and $\eta\times\nu= - \alpha (\xi\times\eta)$. \autoref{eq:dist_coeff_vanish} then reduces to
				\[\inner{p_1}{\xi\times\eta}=\inner{p_2}{\xi\times\eta}=0, \]
			which, however, contradicts the property from \autoref{cond:DT_asym_points_cond2_fixed_dir}. 
	\end{itemize}
	Thus, $\Psi$ cannot satisfy \autoref{eq:dt_dist_cond} and hence has to be DT-asymmetric.
\end{proof}

\subsection{The set of DT-symmetric functions}

As there exist DT-asymmetric functions with arbitrarily small norm, as we have seen in \autoref{thm:dt_def_dist}, the slightest perturbation of the zero function, which is as DT-symmetric as it can be, can make it DT-asymmetric. It therefore seems plausible that any DT-symmetric function can be transformed into a DT-asymmetric function by an arbitrarily small distortion, which would mean that the set $\mathcal S_{\mathrm{DT}}$ only consists of boundary points.

On the other hand, we always have a small neighborhood consisting only of DT-asymmetric functions around every DT-asymmetric function.

\begin{lemma}\label{thm:dt_S_closed}
	The set $\mathcal{S}_{\mathrm{DT}}$ is a closed subset of $\mathcal{O}$.
\end{lemma}
\begin{proof}
We have already observed in \autoref{prop:admissible_obj_closed} that $\mathcal{O}$ is a closed subspace of $L^2(\B^3_{r_{\mathrm s}})$. So, let $(f_j)_{j=1}^\infty\subset\mathcal{S}_{\mathrm{DT}}$ be a converging sequence in $L^2(\B_{r_{\mathrm s}}^3)$ with limit $f\in\mathcal O$. The sequence $(\nabla\hat f_j)_{j=1}^\infty$ of the gradients of the Fourier transforms $\hat f_j$ of the functions $f_j$ then converges uniformly to the gradient $\nabla\hat f$ of the Fourier transform $\hat f$ of $f$, since we have
\[ \sup_{k\in\R^3}\norm{\nabla \hat f_j(k)-\nabla \hat f(k)} \le \frac1{(2\pi)^{\frac32}}\int_{\B^3_{r_{\mathrm s}}}\norm x\abs{f_j(x)-f(x)}\d x \le \frac1{(2\pi)^{\frac32}}\left(\int_{\B^3_{r_{\mathrm s}}}\norm x^2\d x\right)^{\frac12}\norm{f_j-f}_{L^2}. \]

Since the functions $f_j$ are by assumption DT-symmetric, we find according to \autoref{def:dt_symmetry} for each $j\in\N$ vectors $\xi_j,\eta_j,\nu_j\in S^2$ with $\left<\xi_j,\eta_j\right>=0$ such that we have
\begin{equation}
\inner{\nabla\hat f_j(\mu\xi_j+h(\mu)\eta_j)}{(\mu\xi_j+h(\mu)\eta_j)\times\nu_j}=0\text{ for all }\mu\in(-k_0,k_0).\label{eq:dt_cond}
\end{equation}
By the compactness of $S^2$, there exist converging subsequences $(\xi_{j_\ell})_{\ell=1}^\infty$, $(\eta_{j_\ell})_{\ell=1}^\infty$, and $(\nu_{j_\ell})_{\ell=1}^\infty$ with limits $\xi\coloneqq\lim_{\ell\to\infty}\xi_{j_\ell}$, $\eta\coloneqq\lim_{\ell\to\infty}\eta_{j_\ell}$ and $\nu\coloneqq\lim_{\ell\to\infty}\nu_{j_\ell}$.

The corresponding functions $f_{j_\ell}$, $\ell\in\N$, all satisfy \autoref{eq:dt_cond} and the uniform convergence of $\nabla\hat f_{j_\ell}$ implies that we have for every $\mu\in(-k_0,k_0)$ in the limit
\[ 0 = \lim_{\ell\to\infty}\inner{\nabla\hat f_{j_\ell}(\mu\xi_{j_\ell}+h(\mu)\eta_{j_\ell})}{(\mu\xi_{j_\ell}+h(\mu)\eta_{j_\ell})\times\nu_{j_\ell}} = \inner{\nabla\hat f(\mu\xi+h(\mu)\eta)}{(\mu\xi+h(\mu)\eta)\times\nu} \]
so that we also have $f\in\mathcal S_{\mathrm{DT}}$.
\end{proof}

If we can thus indeed show that $\mathcal S_{\mathrm{DT}}=\partial\mathcal S_{\mathrm{DT}}$, we have that $\mathcal S_{\mathrm{DT}}\subset\mathcal O$ is nowhere dense so that the infinitesimal common circle method is guaranteed to give us a unique reconstruction of the rotational motion for a generic object.

\begin{definition}\label{def:generic}
A subset $U$ of a topological space $X$ is called nowhere dense if the closure of $U$ has empty interior. We call a subset $V$ of $X$ generic if its complement $X\setminus V$ is nowhere dense in $X$.
\end{definition}  

To prove this, we will add a suitably constructed DT-asymmetric object composed of finitely many small point masses as in \autoref{section:DT_asym_construction} to a given DT-symmetric function to destroy its DT-symmetry.

\begin{theorem}\label{thm:dt_S_nowheredense}
	The set $\mathcal{S}_{\mathrm{DT}}$ is nowhere dense in $\mathcal{O}$.
\end{theorem}
\begin{proof}
To show that the closed set $\mathcal{S}_{\mathrm{DT}}$ does not have any interior point, we assume by contradiction that there was a DT-symmetric function $f\in\mathcal{S}_{\mathrm{DT}}$ and a radius $\delta>0$ such that the open ball $U_\delta(f)\coloneqq\set{\tilde f\in\mathcal O}{\norm{\tilde f-f}_{L^2}<\delta}$ with radius $\delta$ around $f$ is contained in $\mathcal{S}_{\mathrm{DT}}$. According to \autoref{thm:dt_smooth_approx}, we could then find a parameter $\varepsilon>0$ and a function $g\in C_{\mathrm c}^\infty(\B^3_{r_{\mathrm s}})\cap U_\delta(f)$ with $\supp(g)\subset \B^3_{r_{\mathrm s}-2\varepsilon}$ and select a radius $\tilde\delta>0$ with
		\begin{equation}\label{eq:dt_nested_interior_point}
			U_{\tilde{\delta}}(g)\subset U_\delta(f)\subset\mathcal{S}_{\mathrm{DT}}.
		\end{equation}

	Choosing $\varepsilon$ sufficiently small, we can further find a DT-asymmetric point set $P\subset\R^3$ with some additional properties:
	\begin{enumerate}
		\item \label{dt_prop_sphere} All the points $p\in P$ lie on the sphere $\partial\B_{r_{\mathrm s}-\varepsilon}$ with radius $r_{\mathrm s}-\varepsilon$.
		\item \label{dt_prop_uniform} The points are spread across the surface such that there exists for every direction $u\in S^2$ a point of $P$ in the spherical cap  
		\[ C_\varepsilon(u)\coloneqq\set{x \in \partial \B^3_{r_{\mathrm s}-\varepsilon}}{\inner xu > r_{\mathrm s}-2\varepsilon} \] 
centered at $u$ with height $\varepsilon$.
		\item \label{dt_prop_maximum} The set $P$ satisfies the following maximum projection property: For every choice of directions $\xi,\eta,\nu\in S^2$ with $\inner\xi\eta=0$, there exists a unit vector $u\in \spn\{\xi,\eta\}$ such that we have a corresponding point $p_1(u)\in P\setminus(\spn\{\xi,\eta\}\cup\spn\{\nu\})$ with
\begin{equation}\label{eq:dt_maximum_proj}
\inner{p_1(u)}u>\inner pu\text{ for all }p\in P\setminus\{p_1(u)\}.
\end{equation}
	\end{enumerate}
	Such a point set can be indeed constructed as described in \autoref{rem:dt_asym_points_construction} by iteratively adding points to an existing DT-asymmetric point set in accordance with the conditions established in \autoref{thm:dt_asym_points_existence} until we have a point in every set $C_\varepsilon(u)$, $u\in S^2$. The maximum projection property is then automatically ensured by \autoref{thm:dt_max_proj}.

	Following the construction in \autoref{thm:dt_def_dist}, we pick for the DT-asymmetric point set $P=\{p_j\}_{j=1}^N$ weights $\{w_j\}_{j=1}^N$ in $\R\setminus\{0\}$ with $\sum_{j=1}^{N} w_j p_j=0$ and choose for some $\tilde\varepsilon<\varepsilon$ the indicator function $\psi\coloneqq\bm 1_{\B^3_{\tilde\varepsilon}}$ of the ball with radius $\tilde\varepsilon$. We then consider the function $\Psi_{\tilde\varepsilon}\in\mathcal O$ given by
	\begin{equation}\label{eq:def_of_Psi_eps}
		\Psi_{\tilde\varepsilon}(x)\coloneqq\sum_{j=1}^{N}w_j \bm 1_{\B^3_{\tilde\varepsilon}}(x-p_j),
	\end{equation}
	 which has the same structure as the function in \autoref{eq:dt_def_dist} and is thus according to \autoref{thm:dt_def_dist} DT-asymmetric. Moreover, we make sure to have $\tilde\varepsilon$ selected such that $\norm{\Psi_{\tilde\varepsilon}}_{L^2}\leq\tilde\delta$.
	
We claim now that $g+\Psi_{\tilde\varepsilon}$ is necessarily DT-asymmetric, which will give us a contradiction to \autoref{eq:dt_nested_interior_point} and thus conclude the proof.

We will show this again by contradiction: So let us assume that $g+\Psi_{\tilde\varepsilon}$ was DT-symmetric so that we would have according to \autoref{def:dt_symmetry} some unit vectors $\xi,\eta,\nu\in S^2$ with $\inner{\xi}{\eta}=0$ and
\[\inner{\nabla (\hat g+\hat\Psi_{\tilde\varepsilon})(H(\mu))}{H(\mu)\times\nu}=0\text{ for all }\mu\in(-k_0,k_0),\]
where we define $H(\mu)\coloneqq \mu\xi+h(\mu)\eta$ with $h$ being given by \autoref{eq:dt_def_h}.

Substituting herein the definition of $\Psi_{\tilde\varepsilon}$ from \autoref{eq:def_of_Psi_eps} and calculating the gradient, we would obtain by using that $\inner{\nabla\hat\psi(H(\mu))}{H(\mu)\times\nu}=0$ (as $\psi$ and therefore also $\hat\psi$ are radial functions) the equality
\begin{equation}\label{eq:dt_symmetry_condition}
-\i \hat\psi(H(\mu))\sum_{j=1}^N w_j\inner{p_j}{H(\mu)\times\nu}\e^{-\i\inner{p_j}{H(\mu)}} = -\inner{\nabla\hat g(H(\mu))}{H(\mu)\times\nu}\text{ for all }\mu\in(-k_0,k_0).
\end{equation}

Since the function $h$ holomorphically extends to the complex plane up to a branch cut $\Gamma\subset\C$ between the points $-k_0$ and $k_0$ where the square root in the definition of $h$ vanishes and $\hat\psi$ and $\hat g$ are entire functions (as they are Fourier transforms of compactly supported functions), we can holomorphically extend \autoref{eq:dt_symmetry_condition} to $\C\setminus\Gamma$. We specifically select $\Gamma\coloneqq(-\infty,-k_0]\cup[k_0,\infty)$, which corresponds to using the principal branch of the square root.

Since $\nabla\hat g$ is the Fourier transform of the compactly supported smooth function $\tilde g\in C_{\mathrm c}^\infty(\B_{r_{\mathrm s}}^3)$, $\tilde g(x)\coloneqq-\i xg(x)$, with $\supp(\tilde g)\subset\B_{r_{\mathrm s}-2\varepsilon}^3$, we know from the Paley--Wiener theorem (see \cite[Theorem 7.3.1]{Hoe03}) that there exists a real constant $C>0$ such that we can estimate
		\begin{equation}\label{eq:dt_paley_wiener}
			\abs*{\det\big(\nabla \hat g(H(z)),H(z),\nu\big)}\leq C(1+\norm{H(z)})^{-1}e^{(r_{\mathrm s}-2\varepsilon)\norm{\Im(H(z))}}\text{ for all }z\in\C\setminus\Gamma.
		\end{equation}
		Using this inequality, we would get from \autoref{eq:dt_symmetry_condition} the relation
\begin{equation}\label{eq:dt_symmetry_inequality}
\abs*{\hat\psi(H(z))\sum_{j=1}^N w_jv_j(z)\e^{-\i\inner{\Re(H(z))}{p_j}}\e^{\alpha_j(z)}} \leq \frac C{1+\norm{H(z)}}\text{ for all }z\in\C\setminus\Gamma,
\end{equation}
where we introduced
\[ v_j(z)\coloneqq\det(p_j,H(z),\nu)\text{ and }\alpha_j(z)\coloneqq\inner{p_j}{\Im(H(z))}-(r_{\mathrm s}-2\varepsilon)\norm{\Im(H(z))}. \]

We now choose a direction $u\in S^2\cap\spn\{\xi,\eta\}\setminus\{-\eta,\eta\}$ for which we have, as requested in \autoref{dt_prop_maximum}, a point $p_1(u)\in P\setminus(\spn\{\xi,\eta\}\cup\spn\{\nu\})$ fulfilling \autoref{eq:dt_maximum_proj}. For the sake of a simpler notation, let us assume that the set $P$ is enumerated such that $p_1(u)=p_1$.

According to \autoref{thm:dt_prop_HIm}, we have for every $r\in(k_0,\infty)$ that the function
\[ H_r\colon(-\pi,\pi)\setminus\{0\}\to S^2\cap\spn\{\xi,\eta\}\setminus\{-\eta,\eta\},\;H_r(\varphi)\coloneqq\frac{\Im(H(r\e^{\i\varphi}))}{\norm{\Im(H(r\e^{\i\varphi}))}}, \]
is well-defined and surjective so that we can find a value $\varphi_r\in(-\pi,\pi)\setminus\{0\}$ with $H_r(\varphi_r)=u$. Moreover, we know from \autoref{eq:dt_prop_HIm_asymptotics} that
\[ \lim_{r\to\infty}\norm[\big]{\sin(\varphi_r)\xi+\sgn(\varphi_r)\cos(\varphi_r)\eta-u} = \lim_{r\to\infty}\norm[\big]{\sin(\varphi_r)\xi+\sgn(\varphi_r)\cos(\varphi_r)\eta-H_r(\varphi_r)}=0, \]
which in particular implies that $\lim_{r\to\infty}\sin(\varphi_r)=\inner\xi u\ne0$.

To get the desired contradiction, let us analyze the asymptotic behavior of the left-hand side of \autoref{eq:dt_symmetry_inequality} for $z=r\e^{\i\varphi_r}$ as $r\to\infty$.
\begin{itemize}
\item
For the exponents $\alpha_j$, $j\in\{1,\ldots,N\}$, we find that
\[ \alpha_j(r\e^{\i\varphi_r})=\norm{\Im(H(r\e^{\i\varphi_r}))}\big(\inner{p_j}{H_r(\varphi_r)}-(r_{\mathrm s}-2\varepsilon)\big) = \norm{\Im(H(r\e^{\i\varphi_r}))}\big(\inner{p_j}u-(r_{\mathrm s}-2\varepsilon)\big). \]

And we know from \autoref{dt_prop_uniform} that $\inner{p_1}u>r_{\mathrm s}-2\varepsilon$ so that we get with \autoref{item:dt_prop_H1} from \autoref{thm:dt_prop_H} the limit
\[ \lim_{r\to\infty}\alpha_1(r\e^{\i\varphi_r}) \ge \lim_{r\to\infty}r\abs{\sin(\varphi_r)}\big(\inner{p_1}u-(r_{\mathrm s}-2\varepsilon)\big) = \infty. \]
Moreover, \autoref{eq:dt_maximum_proj} ensures that we have for $j\in\{2,\ldots,N\}$
\[ \lim_{r\to\infty}\left(\alpha_j(r\e^{\i\varphi_r})-\alpha_1(r\e^{\i\varphi_r})\right) \le -\lim_{r\to\infty}r\abs{\sin(\varphi_r)}(\inner{p_1}u-\inner{p_j}u) = -\infty. \]

\item
Plugging in the definition of $H$ and $h$ into the one of $v_j$, $j\in\{1,\ldots,N\}$, we find that
\[ v_j(r\e^{\i\varphi_r}) = r\e^{\i\varphi_r}\inner{p_j}{\xi\times\nu}+\left(\sqrt{k_0^2-r^2\e^{2\i\varphi_r}}-k_0\right)\inner{p_j}{\eta\times\nu}. \]
Taking the limit $r\to\infty$, this becomes
\[ \lim_{r\to\infty}\frac{v_j(r\e^{\i\varphi_r})}{r\e^{\i\varphi_r}}=\inner{p_j}{\xi\times\nu}-\i\sgn(\inner\xi u)\inner{p_j}{\eta\times\nu}, \]
where we use that $\lim_{r\to\infty}\sgn(\sin(\varphi_r))=\sgn(\inner\xi u)$. The limit is nonzero for $j=1$, since $\inner{p_1}{\xi\times\nu}=0=\inner{p_1}{\eta\times\nu}$ would imply that $p_1$ is parallel to $\nu$ which we explicitly excluded. 

\item
Finally, we come to the factor $\hat\psi(H(r\e^{\i\varphi_r}))$, which we can express with the explicit formula for the Fourier transform of the indicator function by
\[ \hat\psi(H(r\e^{\i\varphi_r})) = \sqrt{\frac{2}{\pi}}\frac{\sin\left(\tilde\varepsilon\sigma_r(\varphi_r)\right)-\tilde\varepsilon\sigma_r(\varphi_r)\cos\left(\tilde\varepsilon\sigma_r(\varphi_r)\right)}{\sigma_r^3(\varphi_r)}, \]
where we abbreviated $\sigma_r(\varphi_r)\coloneqq(\sum_{j=1}^3H_j^2(r\e^{\i\varphi_r}))^{\frac12}$.
(The choice of the branch of the square root in $\sigma_r$ is irrelevant because the $\operatorname{sinc}$ function and the cosine are both even.)

We rewrite $\hat\psi(H(r\e^{\i\varphi_r}))$ by expressing the sine and cosine in terms of the exponential function 
	\[\hat\psi(H(r\e^{\i\varphi_r}))=a(r)\e^{\tilde\varepsilon\Im(\sigma_r(\varphi_r))}+b(r)\e^{-\tilde\varepsilon\Im(\sigma_r(\varphi_r))} \]
with the coefficients
\[ a(r)\coloneqq\frac{\i-\tilde{\varepsilon}\sigma_r(\varphi_r)}{\sigma_r^3(\varphi_r)\sqrt{2\pi}}\e^{-\i\tilde\varepsilon\Re(\sigma_r(\varphi_r))}\text{ and }b(r)\coloneqq-\frac{\i+\tilde{\varepsilon}\sigma_r(\varphi_r)}{\sigma_r^3(\varphi_r)\sqrt{2\pi} }\e^{\i\tilde\varepsilon\Re(\sigma_r(\varphi_r))}. \]
From \autoref{item:dt_prop_H4} in \autoref{thm:dt_prop_H}, we know that
\[ \lim_{r\to\infty}\abs*{\frac{\sigma_r(\varphi_r)}{\sqrt{2k_0r}}} = 1 \text{ and }\lim_{r\to\infty}\frac1{\sqrt{2k_0r}}\abs{\Im(\sigma_r(\varphi_r))} = \lim_{r\to\infty}\abs{\Im(\e^{\i(\frac{\varphi_r}2+\sgn(\varphi_r)\frac\pi4)})} \ge\frac1{\sqrt2}, \]
which ensures that the coefficients $a(r)$ and $b(r)$ do not vanish for sufficiently large values of $r\in(k_0,\infty)$ and that one of the summands (depending on the sign of $\Im(\sigma_r(\varphi_r))$) tends to infinity while the other converges to $0$.

Hence, we find that $\lim_{r\to\infty}\abs{\hat\psi(H(r\e^{\i\varphi_r}))} = \infty$.
\end{itemize}

Putting this together, we get that
\begin{align*}
&\lim_{r\to\infty}\abs*{\hat\psi(H(r\e^{\i\varphi_r}))\sum_{j=1}^N w_jv_j(r\e^{\i\varphi_r})\e^{-\i\inner{\Re(H(r\e^{\i\varphi_r}))}{p_j}}\e^{\alpha_j(r\e^{\i\varphi_r})}} \\
&\qquad= \lim_{r\to\infty}\abs*{\hat\psi(H(r\e^{\i\varphi_r}))w_1v_1(r\e^{\i\varphi_r})\e^{\alpha_1(r\e^{\i\varphi_r})}(1+C(r))} = \infty
\end{align*}
because the remaining summands fulfill
\[ C(r)\coloneqq\sum_{j=2}^N \frac{w_jv_j(r\e^{\i\varphi_r})}{w_1v_1(r\e^{\i\varphi_r})}\e^{\i\inner{\Re(H(r\e^{\i\varphi_r}))}{p_1-p_j}}\e^{\alpha_j(r\e^{\i\varphi_r})-\alpha_1(r\e^{\i\varphi_r})} \to 0\;(r\to\infty), \]
which contradicts \autoref{eq:dt_symmetry_inequality}, since the right-hand side therein converges to zero because of \autoref{item:dt_prop_H1} in \autoref{thm:dt_prop_H}.

\end{proof}

\section{Reconstruction for parallel-beam tomography data}\label{section:pb}

We will now turn to the seemingly simpler problem of reconstructing the rotational motion from the measurements in the parallel-beam approximation as derived in \autoref{eq:pb_measurements_Fourier}. We again drop the constant in front.

\begin{definition}
Let $f\in\mathcal{O}$ be an admissible object and $R\colon[0,T]\to\SO(3)$ be a rotational motion. We define the parallel-beam measurements $\hat m$ of $f$ under the rotation $R$ as
\begin{equation}\label{eq:def-measurement-pb}
\hat m\colon[0,T]\times \R^2\to \C,\;\hat m(t,k)\coloneqq \hat f\left(R(t)\begin{pmatrix} k\\0 \end{pmatrix}\right).
\end{equation}
\end{definition}

Although this attempts to model essentially the same experiment as the diffraction tomography measurements, just in a rougher approximation, this further simplification introduces a global symmetry into the problem: The parallel-beam measurements $\hat m$ of an object $f\in\mathcal O$ under a rotational motion $R\colon [0,T]\to\SO(3)$ and the measurements $\hat m_\Sigma$ of the object $f\circ\Sigma$ under the rotational motion $\Sigma R\Sigma$, where $\Sigma\coloneqq\diag(1,1,-1)$ is the reflection at the plane $\set{x\in\R^3}{x_3=0}$, are identical, since we have for all $(t,k)\in[0,T]\times\R^2$ the relation
\begin{equation}\label{eq:equivalence_class}
\hat m(t,k) = \hat f\left(R(t)\begin{pmatrix}k\\0\end{pmatrix}\right) = (\hat f\circ\Sigma)\left(\Sigma R(t)\begin{pmatrix}k\\0\end{pmatrix}\right) = (\hat f\circ\Sigma)\left(\Sigma R(t)\Sigma\begin{pmatrix}k\\0\end{pmatrix}\right) = \hat m_\Sigma(t,k).
\end{equation}
The best we can achieve is therefore a reconstruction of the equivalence class $\{R,\Sigma R\Sigma\}$ of the function $R$.

\begin{problem}\label{pr:pb}
Let $T>0$ be some fixed parameter. Under which conditions on the function $f\in\mathcal O$ and the motion $R\in C^2([0,T];\SO(3))$ with the normalization $R(0)=\I_{3\times3}$, is it possible to uniquely recover the equivalence class $\{R,\Sigma R\Sigma\}$, $\Sigma\coloneqq\diag(1,1,-1)$, from the corresponding parallel-beam measurements $\hat m\colon[0,T]\times\R^2\to\C$?
\end{problem}

The main idea to approach this task is the relation
\begin{equation}\label{eq:common-line}
\hat m\big(t,\lambda P(e_3\times R(t)^\top R(s)e_3)\big) = \hat f\big(\lambda(R(t)e_3)\times(R(s)e_3)\big) = \hat m\big(s,-\lambda P(e_3\times R(s)^\top R(t)e_3)\big)
\end{equation}
for all $\lambda\in\R$, where we write $P(k_1,k_2,k_3)\coloneqq(k_1,k_2)$ for the projection onto the components orthogonal to the illumination direction, 
between measurement data $\hat m(t,\cdot)$ and $\hat m(s,\cdot)$ at two different time steps $s,t\in[0,T]$. This signifies that the values of the function $\hat m(t,\cdot)$ along a one-dimensional subspace coincide with those of the function $\hat m(s,\cdot)$ along a corresponding chosen line through the origin. By detecting this line, its direction provides us directly a condition for the relative orientation $R(t)^\top R(s)$ of the object between the two measurements. This reconstruction approach became famous under the name of the common line method for its application in cryogenic electron microscopy, see \cite{Hee87,Gon88b,SinCoiSigCheShk10,SinShk11}.

In \cite{ElbRitSchSchm20}, we have derived an infinitesimal version of \autoref{eq:common-line}, applicable when the object undergoes a sufficiently smooth motion, by considering the limit $s\to t$. We can then again write \autoref{eq:common-line} as an equation for the angular velocity $\omega\colon[0,T]\to\R$ corresponding to $R$, as introduced in \autoref{def:ang_vel}, which we conveniently represent in cylindrical coordinates $\rho_\omega, \phi_\omega$, and $\zeta_\omega$, where we are using the notation from \autoref{eq:def-cylind}. By expanding \autoref{eq:common-line} in a Taylor polynomial in the variable $s$ around the time $t\in[0,T]$ and equating the coefficients from the first order terms in $(t-s)$, we arrive at the infinitesimal common line equation
	\begin{equation}\label{eq:inf-common-line-D1}
		\partial_t \hat m(t,\lambda \rho_\omega(t) \phi_\omega(t))=\zeta_\omega(t) \inner{\nabla_k \hat m(t,\lambda \rho_\omega(t)\phi_\omega(t))}{\lambda \rho_\omega(t)\phi_\omega^\perp(t)}\text{ for all }\lambda\in\R,
	\end{equation}
for the values $\phi_\omega(t)$ and $\zeta_\omega(t)$ of the azimuth direction and the third component of the angular velocity at the time $t$. For the derivation, we refer to \cite[Proposition 3.7]{ElbRitSchSchm20}.

Although the value $\rho_\omega(t)$ of the cylindrical radius appears in \autoref{eq:inf-common-line-D1}, it can simply be absorbed into the parameter $\lambda$ if $\rho_\omega(t)\neq0$ and we need to consider higher order terms to potentially reconstruct it. If $\rho_\omega(t)=0$, the equation reduces to $\partial_t\hat m(t,0)=0$ which is by the definition of $\hat m$ in \autoref{eq:def-measurement-pb} always true so that we cannot gain any information from \autoref{eq:inf-common-line-D1} in this case.

To reconstruct the cylindrical radius $\rho_\omega$ at some point $t\in[0,T]$ where $\rho_\omega(t)\ne0$, we consider (since the second order terms do not provide any help) the third order terms which yield the equation
	\begin{equation}\label{eq:inf-common-line-D3}
		(\zeta_\omega(t)+\varphi_\omega'(t))\left(a_0(t,\lambda)+a_{02}(t,\lambda)\rho_\omega^2(t)+a_1(t,\lambda)\frac{\rho_\omega'(t)}{\rho_\omega(t)}\right) = 0\text{ for all }\lambda\in\R
	\end{equation}
with the coefficients $a_0,a_{02}$ and $a_1$ depending (besides on the derivatives of $\hat m$) only on the functions $\phi_\omega=({\cos}\circ\varphi_\omega,{\sin}\circ\varphi_\omega)$ and $\zeta_\omega$ which we expect to obtain beforehand from \autoref{eq:inf-common-line-D1}, see \cite[Proposition~3.8]{ElbRitSchSchm20}. These coefficients are explicitly given by
\begin{alignat}{2} 
	&a_0(t,\lambda)
	&&\begin{aligned}[t]
		&\coloneqq\zeta_\omega(t)(\zeta_\omega(t)-\varphi_\omega'(t))\D_k^3\hat m(t,\lambda \phi_\omega(t))\left[\lambda\phi_\omega^\perp(t),\lambda\phi_\omega^\perp(t),\lambda\phi_\omega^\perp(t)\right] \\
		&\qquad+2\D_k^2\hat m(t,\lambda \phi_\omega(t))\left[\lambda\phi_\omega^\perp(t),\lambda\zeta_\omega(t)\varphi_\omega'(t)\phi_\omega(t)-\lambda\zeta_\omega'(t)\phi_\omega^\perp(t)\right] \\
		&\qquad+2\D_k\hat m(t,\lambda \phi_\omega(t))\left[\lambda\zeta_\omega^2(t)\phi_\omega^\perp(t)+\lambda\zeta_\omega'(t)\phi_\omega(t)\right] \\
		&\qquad-(3\zeta_\omega(t)-\varphi_\omega'(t))\partial_t\D_k^2\hat m(t,\lambda \phi_\omega(t))\left[\lambda\phi_\omega^\perp(t),\lambda\phi_\omega^\perp(t)\right] \\
		&\qquad+2\partial_{tt}\D_k\hat m(t,\lambda \phi_\omega(t))\left[\lambda\phi_\omega^\perp(t)\right],
	\end{aligned} \label{eq:inf-cl-a0} \\
	&a_{02}(t,\lambda) &&\coloneqq 2\D_k\hat m(t,\lambda \phi_\omega(t))\left[\lambda\phi_\omega^\perp(t)\right], \label{eq:inf-cl-a02} \\
	&a_1(t,\lambda)
	&&\begin{aligned}[t]
		&\coloneqq 2\Big(\zeta_\omega(t)\D_k^2\hat m(t,\lambda \phi_\omega(t))\left[\lambda\phi_\omega^\perp(t),\lambda \phi_\omega^\perp(t)\right] \\
		&\qquad-\zeta_\omega(t)\D_k\hat m(t,\lambda \phi_\omega(t))\left[ \lambda \phi_\omega(t)\right]-\partial_t\D_k\hat m(t,\lambda \phi_\omega(t))\left[\lambda \phi_\omega^\perp(t)\right]\Big),
	\end{aligned} \label{eq:inf-cl-a1}
\end{alignat}
where we denote by $\D_k^\ell\hat m(t,\kappa)\colon(\R^2)^\ell\to\C$ the derivative of order $\ell\in\N$ of the function $k\mapsto \hat m(t,k)$ at a point $\kappa\in\R^2$ for fixed $t\in[0,T]$ and write the evaluation of the tensor $\D_k^\ell\hat m(t,\kappa)$ at the vectors $(v_j)_{j=1}^\ell\in(\R^2)^\ell$ in the form $\D_k^\ell \hat m(t,\kappa)[(v_j)_{j=1}^\ell]$.

Still, a reconstruction of the value $\rho_\omega(t)$ of the cylindrical radius from \autoref{eq:inf-common-line-D3} will be unfeasible at a time $t\in[0,T]$ where $\zeta_\omega(t)+\varphi_\omega'(t)=0$.

\subsection{Reconstructible rotational motions} 

We will simply ignore the scenarios where \autoref{eq:inf-common-line-D1} or \autoref{eq:inf-common-line-D3} become trivial by considering only time steps $t\in[0,T]$ where $\rho_\omega(t)\ne0$ and $\zeta_\omega(t)+\varphi_\omega'(t)\ne0$. To make this condition more apparent, we will formulate it in a geometrically more intuitive way.

\begin{definition}\label{def:pb_degmotion}
We call a rotational motion $R\in C^2([0,T];\SO(3))$ nondegenerate at a point $t\in[0,T]$ if the vectors $R(t) e_3$, $R'(t) e_3$, and $R''(t)e_3$ are linearly independent.
\end{definition}

\begin{lemma}\label{th:nondeg_D3}
	Let $R\in C^2([0,T];\SO(3))$ be a rotational motion and $\omega$ be its corresponding angular velocity written in cylindrical coordinates according to \autoref{eq:def-cylind}. Then, $R$ is nondegenerate at a point $t\in[0,T]$ if and only if
	\begin{equation}\label{eq:nondeg_D3}
		\zeta_\omega(t)+\varphi_\omega'(t)\ne0 \text{ and }\rho_\omega(t)\ne0.
	\end{equation}
\end{lemma}
\begin{proof}
	We can write the nondegeneracy of $R$ at $t$ as
	\[ 0\ne\inner{R(t) e_3}{(R'(t) e_3)\times(R''(t) e_3)} = \inner{e_3}{(R(t)^\top R'(t) e_3)\times(R(t)^\top R''(t) e_3)}. \]
	Since we have by the definition of $\omega$, see \autoref{eq:def-omega}, that
	\[ R''(t)x = R'(t)(\omega(t)\times x)+R(t)(\omega'(t)\times x) = R(t)\big(\omega(t)\times(\omega(t)\times x)+\omega'(t)\times x\big) \]
	for every $x\in\R^3$, this becomes
	\[ \inner*{e_3}{(\omega(t)\times e_3)\times\big(\omega(t)\times(\omega(t)\times e_3)+\omega'(t)\times e_3\big)} \ne 0. \]
Applying the vector identity $a\times(b\times c)=\left<a,c\right>b-\left<a,b\right>c$, we get that $(\omega(t)\times e_3)\times(\omega(t)\times(\omega(t)\times e_3))=\norm{\omega(t)\times e_3}^3\omega(t)$ and $(\omega(t)\times e_3)\times(\omega'(t)\times e_3)=-\inner{\omega(t)\times e_3}{\omega'(t)}e_3$ so that the nondegeneracy at the time $t$ is equivalent to
	\[ \norm{\omega(t)\times e_3}^2\inner{e_3}{\omega(t)}+\inner[\big]{e_3\times\omega(t)}{\omega'(t)} \ne 0. \]
	Expressing $\omega$ herein in cylindrical coordinates as defined in \autoref{eq:def-cylind}, this takes with $e_3\times\omega(t)=\rho_\omega(t)(-\sin(\varphi_\omega(t)),\cos(\varphi_\omega(t)),0)$ the form
	\begin{equation*}
		\rho^2_\omega(t)(\zeta_\omega(t)+\varphi_\omega'(t))\ne0,
	\end{equation*}
	which is equivalent to \autoref{eq:nondeg_D3}.
\end{proof}

\begin{remark}
This condition is tightly connected to the one from the (discrete) common line method using \autoref{eq:common-line}. There, we have the issue that the potentially obtainable vectors $P(e_3\times R(t)^\top R(s)e_3)$ and $P(e_3\times R(s)^\top R(t)e_3)$ are not enough to calculate the full matrix $R(t)^\top R(s)$. This can be remedied by looking at three different time steps $t_1,t_2,t_3\in[0,T]$. The six vectors $P(e_3\times R(t_i)^\top R(t_j)e_3)$ with different indices $i,j\in\{1,2,3\}$ are then sufficient to recover the two relative rotations $R(t_2)^\top R(t_1)$ and $R(t_3)^\top R(t_1)$ provided that the vectors $(R(t_j)e_3)_{j=1}^3$ are not linearly dependent, see \cite[Section 3.4.1]{Fra06}.

Now, if the rotationial motion $R\in C^2([0,T];\SO(3))$ is nondegenerate at a time $t\in[0,T]$, we can ensure by choosing three different points $t_j\in[0,T]$ for some sufficiently small $\delta\in(0,\infty)$ such that $\abs{t_j-t}\in[\delta,2\delta]$ for all $j\in\{1,2,3\}$ with the Taylor expansion
\[R(t_j)e_3=R(t)e_3+(t_j-t) R'(t)e_3+\frac{1}{2}(t_j-t)^2 R''(t)e_3+o(\delta^2)\text{ for all }j\in\{1,2,3\}\]
that
\[ \det\big(R(t_1)e_3,R(t_2)e_3,R(t_3)e_3\big)=C\det\big(R(t)e_3,R'(t)e_3,R''(t)e_3\big)+o(\delta^3)\ne 0, \]
where the constant $C$ is given by $C\coloneqq\frac12\det\big((t_j-t)^{i-1}\big)_{i,j=1}^3=\frac12(t_2-t_1)(t_3-t_1)(t_3-t_2)$. Hence, the nondegeneracy of $R$ guarantees the existence of three points $t_1,t_2,t_3\in[0,T]$ for which the vectors $(R(t_j)e_3)_{j=1}^3$ are linearly independent as required for the common line method.
\end{remark}

The prototypical example of a degenerate rotation is one for which $e_3$ is rotated around a single axis perpendicular to the imaging direction.
\begin{example}
Taking an arbitrary direction $v\in S^2$ with $\inner v{e_3}=0$ and two real-valued functions $\rho,\zeta\in C^2([0,T];\R)$ with $\rho(0)=\zeta(0)=0$, we consider a rotation of the form
\begin{equation}\label{eq:firstorder_rotation}
R\colon[0,T]\to\SO(3),\;R(t)\coloneqq \bm{R}_v(\rho(t))\bm{R}_{e_3}(\zeta(t)),
\end{equation}
where we denote by $\bm{R}_w(\alpha)\in \SO(3)$ the rotation around a vector $w\in S^2$ by the angle $\alpha\in\R$, defined by Rodrigues' rotation formula as
\begin{equation}\label{eq:rodrigues}
\bm{R}_w(\alpha)x \coloneqq \inner wxw+\sin(\alpha)\,w\times x-\cos(\alpha)\,w\times(w\times x)\text{ for all }x\in\R^3.
\end{equation}

Then, the motion $R$ is degenerate at every point $t\in[0,T]$. Indeed, remarking that $\bm{R}_w'(\alpha)x=\bm R_w(\alpha)(w\times x)$, we find that
\begin{align*}
&\det\big(R(t)e_3,R'(t)e_3,R''(t)e_3\big) \\
&\qquad= (\rho')^3(t)\det\big(\bm R_v(\rho(t))e_3,\bm R_v(\rho(t))(v\times e_3),\bm R_v(\rho(t))(v\times(v\times e_3))\big) = 0,
\end{align*}
since $v\times(v\times e_3)=-e_3$.
\end{example}

This limitation to non-degenerate motions might, of course, be only a limitation of the applied common line method and not of the problem itself. We want to give at least an indication that this non-uniqueness of the reconstruction results for the degenerate rotation $R$ considered in \autoref{eq:firstorder_rotation} arises from an intrinsic symmetry inherent in the data.

\begin{lemma}\label{thm:firstorder_nonuniq}
	Let $f\in\mathcal O\cap C^\infty_{\mathrm c}(\B^3_{r_{\mathrm s}})$, $v\in S^2$ with $\inner v{e_3}=0$, and $\rho,\tilde\rho,\zeta\in C^1([0,T];\R)$ be three real-valued, strictly increasing functions with $\rho(0)=\tilde\rho(0)=\zeta(0)=0$. We consider the two rotations $R,\tilde R\in C^2([0,T];\SO(3))$ given by
\[ R(t)\coloneqq \bm{R}_v(\rho(t))\bm{R}_{e_3}(\zeta(t))\text{ and }
\tilde R(t)\coloneqq \bm{R}_v(\tilde\rho(t))\bm{R}_{e_3}(\zeta(t)). \]

We then find a parameter $\varepsilon\in(0,T)$ and a function $g\in L^2(\R^3)$ such that the parallel-beam measurements $\hat m$ of $f$ under the rotation $R$ fulfill
\begin{equation}\label{eq:nonuniq}
\hat m(t,k)=\hat f\left(R(t)\begin{pmatrix} k\\0 \end{pmatrix}\right)=\hat g\left(\tilde R(t)\begin{pmatrix} k\\0 \end{pmatrix}\right)\text{ for all }t\in[0,\varepsilon),\;k\in\R^2.
\end{equation}
\end{lemma}
\begin{proof}
	We want to try to construct a function $g$ fulfilling \autoref{eq:nonuniq} whose Fourier transform has the form $\hat g=\hat f\circ\Phi$ for some function $\Phi\colon\R^3\to\R^3$. \autoref{eq:nonuniq} then becomes with the orthonormal bases $(r_j(t))_{j=1}^3$ and $(\tilde r_j(t))_{j=1}^3$ defined by $r_j(t)\coloneqq R(t) e_j$ and $\tilde r_j(t)\coloneqq\tilde R(t) e_j$ the condition
	\[ \hat f(k_1r_1(t)+k_2r_2(t)) = \hat f(\Phi(k_1\tilde r_1(t)+k_2\tilde r_2(t)))\text{ for all }t\in[0,\varepsilon),\;(k_1,k_2)\in\R^2. \]
	Since we want to have this for a rather arbitrary function $\hat f$, we simply impose that
	\begin{equation}\label{eqPhi}
		k_1r_1(t)+k_2r_2(t) = \Phi(k_1\tilde r_1(t)+k_2\tilde r_2(t))\text{ for all }t\in[0,\varepsilon),\;(k_1,k_2)\in\R^2.
	\end{equation}
	
	Such a function $\Phi$ would clearly have to be linear on all the subspaces 
	\[ E(t) \coloneqq \spn\{\tilde r_1(t),\tilde r_2(t)\} = \set{\kappa\in\R^3}{\inner\kappa{\tilde r_3(t)}=0} = \set{\kappa\in\R^3}{\inner\kappa{\bm R_v(\tilde\rho(t))e_3}=0},\;t\in[0,\varepsilon), \]
	and be uniquely defined on the union of these subspaces by \autoref{eqPhi}.
	
	However, for it to exist, we need to make sure that the thus defined values of $\Phi$ agree on all the intersections $E(t_1)\cap E(t_2)$ for all $t_1,t_2\in[0,\varepsilon)$. But since $v$ is orthogonal to $e_3$, it is for every $t\in[0,\varepsilon)$ also orthogonal to $\bm R_v(\tilde\rho(t)) e_3$, and every plane $E(t)$ therefore contains the line $\spn\{v\}$. We thus have for sufficiently small $\varepsilon>0$ (such that the planes are not identical) that the intersection is for all different values $t_1,t_2\in[0,\varepsilon)$ just the one-dimensional subspace
	\[ E(t_1)\cap E(t_2)=\spn\{v\}. \]
	
	If we express the direction $v$ in the basis $(r_j(t))_{j=1}^3$, we get
\[ v =\sum_{j=1}^2\inner{r_j(t)}vr_j(t) = \sum_{j=1}^2\inner{\bm R_v(\rho(t))\bm R_{e_3}(\zeta(t))e_j}vr_j(t) = \sum_{j=1}^2\inner{\bm R_{e_3}(\zeta(t))e_j}vr_j(t), \]
	and in the same way, the expansion in the basis $(\tilde r_j(t))_{j=1}^3$ is seen to be
\[ v = \sum_{j=1}^2\inner{\bm R_{e_3}(\zeta(t))e_j}v\tilde r_j(t). \]
	
	This means that we have for all $k_1,k_2\in\R$ with $k_1\tilde r_1(t)+k_2\tilde r_2(t)=\mu v$ for some $\mu\in\R$
	\[ k_1r_1(t)+k_2r_2(t) = \mu v = k_1\tilde r_1(t)+k_2\tilde r_2(t) \]
	so that $\Phi(\kappa) = \kappa$ for all $\kappa\in\spn\{v\}$ is a choice consistent with the definition in every plane $E(t)$, $t\in[0,\varepsilon)$.
	
	There therefore exists a unique function $\Phi\colon\bigcup_{t\in[0,\varepsilon)}E(t)\to\R^3$ fulfilling \autoref{eqPhi}. Furthermore, it is clear from \autoref{eqPhi} that $\Phi$ is continuous and $\norm{\Phi(\kappa)}=\norm\kappa$ for all $\kappa\in\bigcup_{t\in[0,\varepsilon)}E(t)$. We can thus extend $\Phi$ to $\R^3$, where we may conserve these properties so that the function $\hat g\coloneqq\hat f\circ\Phi$ is square integrable (since $\hat f$ is an analytic function which decays faster than any polynomial at infinity according to the Paley--Wiener theorem, see \cite[Theorem 7.3.1]{Hoe03}) and we can define $g$ as its inverse Fourier transform.
\end{proof}

Let us stress that this does not prove non-uniqueness for degenerate rotations, since the Fourier transform $\hat g$ of the function $g$ cannot be guaranteed to be analytic so that $g$ has not necessarily compact support and may thus not be in the set $\mathcal O$ of admissible functions.

\subsection{Objects allowing a unique reconstruction} \label{section:pb-uniqueness-symmetry} 

Even in the case of a nondegenerate true rotation, a reconstruction of the rotation from \autoref{eq:inf-common-line-D1} and \autoref{eq:inf-common-line-D3} is not necessarily possible, since certain symmetries of the object may lead to multiple solutions of these equations. However, there is a geometric criterion which characterizes exactly those objects for which a unique reconstruction of the equivalence class of the motion is possible with this method.

\begin{definition}\label{def:pb-symmetry}
We call a function $f\in\mathcal{O}$ PB-symmetric if there exist directions $\xi,\eta\in S^2$ with $\inner\xi\eta=0$ such that
	\begin{equation}\label{eq:def-pbasym}
		\inner{\nabla \hat f(\lambda\xi)}\eta=0\text{ for all }\lambda\in\R.
	\end{equation}
We call $f$ PB-asymmetric if it is not PB-symmetric and denote by $\mathcal{S}_{\mathrm{PB}}$ the set of admissible PB-symmetric objects. 
\end{definition}
\begin{remark}
	The DT-symmetry condition in \autoref{eq:dt_symmetric} simplifies to the PB-symmetry condition in \autoref{eq:def-pbasym} as $k_0\to\infty$. Indeed, we get in the limit $k_0\to\infty$ for the function $h$, introduced in \autoref{eq:dt_def_h}, at every point $\mu\in\R$ that
	\[ h(\mu)=k_0\left(\sqrt{1-k_0^{-2}\mu^2}-1\right)\to0 \]
	and therefore, because of the smoothness of $\hat f$, that
	\[\inner{\nabla\hat f(\mu\xi+h(\mu)\eta)}{(\mu\xi+h(\mu)\eta)\times\nu}\to\inner{\nabla\hat f(\mu\xi)}{\mu\xi\times\nu} \]
	for arbitrary vectors $\xi,\eta,\nu\in S^2$ with $\inner\xi\eta=0$.

	In particular, since a mirror symmetric object satisfies the DT-symmetry condition for every $k_0>0$, as we have seen in \autoref{ex:dt_mirror_sym}, it is also PB-symmetric.
\end{remark}

The PB-symmetry condition can also be formulated in the spatial domain for $f$ directly.

\begin{lemma}
A function $f\in\mathcal O$ is PB-symmetric if and only if there exist directions $\xi,\eta\in S^2$ with $\inner\xi\eta=0$ such that the function $F\colon\R^3\setminus\{0\}\to\C$, whose values
\[ F(w)\coloneqq\int_{\mathbf E_w}xf(x)\d\mathrm S(x) \]
are the centers of $f$ on the planes $\mathbf E_w\coloneqq\set{x\in\R^3}{\inner{x-w}w=0}$ through $w$ orthogonal to the vector $w\in\R^3\setminus\{0\}$, satisfies
\begin{equation}\label{eq:def-pbasym-spatial}
\inner{F(\lambda\xi)}\eta = 0\text{ for almost all }\lambda\in\R\setminus\{0\}.
\end{equation}
\end{lemma}

\begin{proof}
Let $\xi,\eta\in S^2$ be two orthonormal vectors. We then find with $\inner x\xi=\mu$ for all $x\in\mathbf E_{\mu\xi}$ that
\begin{align*}
\inner{\nabla \hat f(\lambda\xi)}{\eta} &= -\frac{\i}{\sqrt[3]{2\pi}} \int_{\R^3}\inner{x}{\eta}f(x)\e^{-\i\lambda\inner{x}{\xi}}\d x \\
&= -\frac{\i}{\sqrt[3]{2\pi}}\int_\R\int_{\mathbf E_{\mu\xi}}\inner{x}{\eta}f(x)\e^{-\i\lambda\inner{x}{\xi}}\d\mathrm S(x)\d\mu
= -\frac{\i}{\sqrt[3]{2\pi}}\int_\R\inner{F(\mu\xi)}\eta\e^{-\i\lambda\mu}\d\mu,
\end{align*}
which means that $\lambda\mapsto\inner{\nabla \hat f(\lambda\xi)}{\eta}$ is the one-dimensional Fourier transform of $\mu\mapsto-\frac\i{2\pi}\inner{F(\mu)}\eta$. In particular, the condition
\[ \inner{\nabla \hat f(\lambda\xi)}{\eta} = 0\text{ for all }\lambda\in\R \]
is therefore equivalent to \autoref{eq:def-pbasym-spatial}.
\end{proof}

Geometrically, this signifies that an object is PB-symmetric if and only if there exist two orthogonal directions $\xi$ and $\eta$ such that the center of mass of every slice taken along $\xi$ is in the linear subspace orthogonal to $\eta$.

We can now verify that the PB-asymmetry is precisely the condition we need to ensure that the first- and third-order infinitesimal common line equations provide a unique reconstruction of the equivalence class of the rotation.

\begin{proposition}\label{thm:pb_uniqueness_firstorder}
Let $f\in\mathcal{O}\setminus\mathcal{S}_{\mathrm{PB}}$ be a PB-asymmetric function and $\hat m$ be the parallel-beam measurements of $f$ under the (not necessarily nondegenerate) rotation $R\in C^1([0,T];\SO(3))$ with associated angular velocity $\omega$ written in the cylindrical coordinates $\rho_\omega, \phi_\omega$, and $\zeta_\omega$ as defined in \autoref{eq:def-cylind}.

We now consider for some $t\in[0,T]$ the first order infinitesimal common line equation
	\begin{equation}\label{eq:PB-uniq_firstOrd_eq}
		\partial_t \hat m(t,\lambda\phi_u)=\zeta_u \inner{\nabla_k \hat m(t,\lambda\phi_u)}{\lambda\phi_u^\perp}\text{ for all }\lambda\in\R
	\end{equation}
for the vector $u\in\R^3$ written in the cylindrical coordinates $\rho_u, \phi_u$, and $\zeta_u$.

\begin{enumerate}
\item
If $\rho_\omega(t)\neq0$, then a vector $u\in\R^3$ is a solution of \autoref{eq:PB-uniq_firstOrd_eq} if and only if
	\[ \zeta_u=\zeta_\omega(t)\text{ and }\inner{\phi_u}{\phi_\omega^\perp(t)}=0. \]
\item
If $\rho_\omega(t)=0$, then a vector $u\in\R^3$ is a solution of \autoref{eq:PB-uniq_firstOrd_eq} if and only if
	\[ \zeta_u=\zeta_\omega(t). \]
\end{enumerate}
\end{proposition}
\begin{proof}
Substituting the definition of $\hat m$ from \autoref{eq:def-measurement-pb} into \autoref{eq:PB-uniq_firstOrd_eq}, we get
\[ \inner*{\nabla\hat f\left(\lambda R(t)\begin{pmatrix}\phi_u\\ 0\end{pmatrix}\right)}{R'(t)\begin{pmatrix}\phi_u\\ 0\end{pmatrix}} = \zeta_u\inner*{\nabla\hat f\left(\lambda R(t)\begin{pmatrix}\phi_u\\ 0\end{pmatrix}\right)}{R(t)\begin{pmatrix}\phi_u^\perp\\ 0\end{pmatrix}}. \]
We then plug in the definition of the angular velocity $\omega$, see \autoref{eq:def-omega}, and obtain
\[ \inner*{\nabla\hat f\left(\lambda R(t)\begin{pmatrix}\phi_u\\ 0\end{pmatrix}\right)}{R(t)\left(\omega(t)\times \begin{pmatrix}\phi_u\\ 0\end{pmatrix}-\zeta_u \begin{pmatrix}\phi_u^\perp\\ 0\end{pmatrix}\right)} = 0. \]
Writing $\omega$ in cylindrical coordinates, this becomes
\begin{equation}\label{eq:PB-uniq_firstOrd__fourier_eq}
	\inner*{\nabla\hat f\left(\lambda R(t)\begin{pmatrix}\phi_u\\ 0\end{pmatrix}\right)}{\rho_\omega(t)\inner{\phi_u}{\phi_\omega^\perp(t)}R(t)e_3+(\zeta_\omega(t)-\zeta_u)R(t)\begin{pmatrix}\phi_u^\perp\\ 0\end{pmatrix}} = 0,
\end{equation}
which shows that every $u\in\R^3$ with $\zeta_u=\zeta_\omega(t)$ and $\rho_\omega(t)\inner{\phi_u}{\phi_\omega^\perp(t)}=0$ is a solution of \autoref{eq:PB-uniq_firstOrd_eq}.

To see that these are all the possible solutions, we introduce the orthonormal basis $(\tilde e_j)_{j=1}^3\subset\R^3$ by
\[ \tilde e_1\coloneqq R(t)\begin{pmatrix}\phi_u\\ 0\end{pmatrix},\;\tilde e_2\coloneqq R(t)\begin{pmatrix}\phi_u^\perp\\ 0\end{pmatrix},\text{ and }\tilde e_3\coloneqq R(t) e_3 \]
so that \autoref{eq:PB-uniq_firstOrd__fourier_eq} takes the form
\begin{equation}\label{eq:PB-uniq_firstOrd_fourier_simpl}
	\inner{\nabla\hat f(\lambda \tilde e_1)}{\rho_\omega(t)\inner{\phi_u}{\phi_\omega^\perp(t)}\tilde e_3+(\zeta_\omega(t)-\zeta_u)\tilde e_2} = 0\text{ for all }\lambda\in\R.
\end{equation}
So, if we had $\zeta_u\ne\zeta_\omega(t)$ or $\rho_\omega(t)\inner{\phi_u}{\phi_\omega^\perp(t)}\ne0$, we would have found with
\[ \eta\coloneqq\rho_\omega(t)\inner{\phi_u}{\phi_\omega^\perp(t)}\tilde e_3+(\zeta_\omega(t)-\zeta_u)\tilde e_2 \]
a nonzero vector perpendicular to $\tilde e_1$ such that $\inner{\hat f(\lambda\tilde e_1)}\eta = 0$ for all $\lambda\in\R$, implying that $f$ is PB-symmetric, which contradicts our assumptions.
\end{proof}

This means that, under the a priori assumption that the measurements $\hat m$ belong to a PB-asymmetric function $f$, we have that the equation
\[ \partial_t \hat m(t,\lambda\phi)=\zeta\inner{\nabla_k \hat m(t,\lambda\phi)}{\lambda\phi^\perp}\text{ for all }\lambda\in\R \]
either
\begin{enumerate}
\item
has a unique solution $(\zeta,\phi)$ in the space $\R\times S^1_+$, $S^1_+\coloneqq\set{(\cos(\alpha),\sin(\alpha))}{\alpha\in[0,\pi)}$, in which case we know that the angular velocity $\omega$ of the rotation $R$ is at the time $t\in[0,T]$ of the form $\omega(t)=(\rho\phi,\zeta)$ for some $\rho\in\R\setminus\{0\}$; or
\item
its solution set in $\R\times S^1_+$ is of the form $\{\zeta\}\times S^1_+$ for some $\zeta\in\R$, in which case $\omega=(0,0,\zeta)$.
\end{enumerate}

In the first case, it thus remains to recover the cylindrical radius $\rho_\omega(t)$ of the angular velocity at the time~$t$.
To do so, we want to use \autoref{eq:inf-common-line-D3}, where the first factor is known to be nonzero if we make the assumption that $R$ is not degenerate, see \autoref{th:nondeg_D3}, and the coefficients $a_0$, $a_{02}$, and $a_1$ are already fully determined from solving the first order equation. We thus solve at fixed $t\in[0,T]$ the equation system
\begin{equation}\label{eq:uniq_ThirdOrd_System}
	a_{02}(t,\lambda)X_1+a_1(t,\lambda)X_2 = -a_0(t,\lambda),\;\lambda\in\R,
\end{equation}
for the variables $X_1=\rho_\omega^2(t)$ and $X_2=\frac{\rho_\omega'(t)}{\rho_\omega(t)}$. To ensure a unique solution, we must guarantee the existence of two values $\lambda_1, \lambda_2 \in \mathbb{R}$ such that the two vectors $(a_{02}(t, \lambda_1), a_1(t, \lambda_1))$ and $(a_{02}(t, \lambda_2), a_1(t, \lambda_2))$ are linearly independent.

\begin{lemma}\label{th:pb-uniq_thirdOrd_coeff}
Let $t\in[0,T]$, $f\in\mathcal{O}$ and $\hat m$ be the parallel-beam measurements of $f$ under the rotation $R\in C^2([0,T];\SO(3))$ with associated angular velocity $\omega$ written in cylindrical coordinates $\rho_\omega, \phi_\omega=({\cos}\circ\varphi_\omega,{\sin}\circ\varphi_\omega)$ and $\zeta_\omega$ as defined in \autoref{eq:def-cylind}. Using the orthonormal basis $(\tilde e_j(t))_{j=1}^3$, defined by
\begin{equation}\label{eq:thirdOrd_ONB}
\tilde e_1(t)\coloneqq R(t)\begin{pmatrix}\phi_\omega(t)\\ 0\end{pmatrix},\;\tilde e_2(t)\coloneqq R(t)\begin{pmatrix}\phi_\omega^\perp(t)\\ 0\end{pmatrix},\;\tilde e_3(t)\coloneqq R(t)e_3,
\end{equation}
the functions $a_0$, $a_{02}$, and $a_1$, given by \autoref{eq:inf-cl-a0}, \autoref{eq:inf-cl-a02}, and \autoref{eq:inf-cl-a1}, can be expressed in terms of $f$ by
\begin{alignat}{2}
&a_0(t,\lambda) &&= 2\rho_\omega'(t)\inner{\nabla\hat f(\lambda\tilde e_1(t))}{\lambda\tilde e_3(t)}-2\rho_\omega^2(t)\inner{\nabla\hat f(\lambda\tilde e_1(t))}{\lambda\tilde e_2(t)}, \label{eq:uniq-inf-cl-a0} \\
&a_{02}(t,\lambda) &&= 2\inner{\nabla\hat f(\lambda\tilde e_1(t))}{\lambda\tilde e_2(t)},\text{ and} \label{eq:uniq-inf-cl-a02} \\
&a_1(t,\lambda) &&= -2\rho_\omega(t)\inner{\nabla\hat f(\lambda\tilde e_1(t))}{\lambda\tilde e_3(t)}. \label{eq:uniq-inf-cl-a1}
\end{alignat}
\end{lemma}
\begin{proof}
In preparation for the calculations, we remark that according to the definition of $\hat m$ in \autoref{eq:def-measurement-pb}, we have for all $k,v,v_j\in\R^2$, $j\in\{1,\ldots,\ell\}$, $\ell\in\N$, that
\begin{align}
&\D_k^\ell\hat m(t,k)[(v_j)_{j=1}^\ell] = \D^\ell\hat f\left(R(t)\begin{pmatrix}k\\0\end{pmatrix}\right)\left[\left(R(t)\begin{pmatrix}v_j\\0\end{pmatrix}\right)_{j=1}^\ell\right], \label{eq:Dm} \\
&\partial_t\D_k^\ell\hat m(t,k)[(v)_{j=1}^\ell]
\begin{aligned}[t]
&=\D^{\ell+1}\hat f\left(R(t)\begin{pmatrix}k\\0\end{pmatrix}\right)\left[R'(t)\begin{pmatrix}k\\0\end{pmatrix},\left(R(t)\begin{pmatrix}v\\0\end{pmatrix}\right)_{j=1}^\ell\right] \\
&\qquad+\ell\,\D^\ell\hat f\left(R(t)\begin{pmatrix}k\\0\end{pmatrix}\right)\left[R'(t)\begin{pmatrix}v\\0\end{pmatrix},\left(R(t)\begin{pmatrix}v\\0\end{pmatrix}\right)_{j=2}^\ell\right],\text{ and}
\end{aligned} \label{eq:DtDm} \\
&\partial_{tt}\D_k\hat m(t,k)[v]
\begin{aligned}[t]
&=\D^3\hat f\left(R(t)\begin{pmatrix}k\\0\end{pmatrix}\right)\left[R'(t)\begin{pmatrix}k\\0\end{pmatrix},R'(t)\begin{pmatrix}k\\0\end{pmatrix},R(t)\begin{pmatrix}v\\0\end{pmatrix}\right] \\
&\qquad+\D^2\hat f\left(R(t)\begin{pmatrix}k\\0\end{pmatrix}\right)\left[R''(t)\begin{pmatrix}k\\0\end{pmatrix},R(t)\begin{pmatrix}v\\0\end{pmatrix}\right] \\
&\qquad+2\,\D^2\hat f\left(R(t)\begin{pmatrix}k\\0\end{pmatrix}\right)\left[R'(t)\begin{pmatrix}k\\0\end{pmatrix},R'(t)\begin{pmatrix}v\\0\end{pmatrix}\right] \\
&\qquad+\D\hat f\left(R(t)\begin{pmatrix}k\\0\end{pmatrix}\right)\left[R''(t)\begin{pmatrix}v\\0\end{pmatrix}\right].
\end{aligned} \label{eq:DttDm}
\end{align}
Moreover, the derivatives of $R$ at the vectors $(\phi_\omega(t),0)$ and $(\phi_\omega^\perp(t),0)$ can be written as
\begin{align}
&R'(t)\begin{pmatrix}\phi_\omega(t)\\ 0\end{pmatrix}= R(t)\left(\omega(t)\times \begin{pmatrix}\phi_\omega(t)\\ 0\end{pmatrix}\right) = \zeta_\omega(t)\tilde e_2(t)\text{ and} \label{eq:DR1} \\
&R'(t)\begin{pmatrix}\phi_\omega^\perp(t)\\ 0\end{pmatrix} = R(t)\left(\omega(t)\times \begin{pmatrix}\phi_\omega^\perp(t)\\ 0\end{pmatrix}\right) = -\zeta_\omega(t)\tilde e_1(t)+\rho_\omega(t)\tilde e_3(t), \label{eq:DR2}
\end{align}
and we similarly get for the second derivatives, writing $\phi_\omega(t)=(\cos(\varphi_\omega(t)),\sin(\varphi_\omega(t)))$ so that we have $\inner{\phi_\omega(t)}{{\phi_\omega'}^\perp(t)}=-\varphi_\omega'(t)$,
\begin{align}
&R''(t)\begin{pmatrix}\phi_\omega(t)\\ 0\end{pmatrix}
\begin{aligned}[t]
&=R(t)\left(\omega(t)\times\left(\omega(t)\times\begin{pmatrix}\phi_\omega(t)\\0\end{pmatrix}\right)+\omega'(t)\times\begin{pmatrix}\phi_\omega(t)\\0\end{pmatrix}\right) \\
&=-\zeta_\omega^2(t)\tilde e_1(t)+\zeta_\omega'(t)\tilde e_2(t)+\rho_\omega(t)\big(\zeta_\omega(t)-\varphi_\omega'(t)\big)\tilde e_3(t)\text{ and}
\end{aligned} \label{eq:DDR1} \\
&R''(t)\begin{pmatrix}\phi_\omega^\perp(t)\\ 0\end{pmatrix}
\begin{aligned}[t]
&=R(t)\left(\omega(t)\times\left(\omega(t)\times\begin{pmatrix}\phi_\omega^\perp(t)\\0\end{pmatrix}\right)+\omega'(t)\times\begin{pmatrix}\phi_\omega^\perp(t)\\0\end{pmatrix}\right) \\
&=-\zeta_\omega'(t)\tilde e_1(t)-(\rho_\omega^2(t)+\zeta_\omega^2(t))\tilde e_2(t)+\rho_\omega'(t)\tilde e_3.
\end{aligned} \label{eq:DDR2}
\end{align}

With \autoref{eq:Dm}, we then find from \autoref{eq:inf-cl-a02} directly the \autoref{eq:uniq-inf-cl-a02}.

Similarly, we get by using \autoref{eq:Dm}, \autoref{eq:DtDm}, and \autoref{eq:DR1} in \autoref{eq:inf-cl-a1} that
\begin{align*}
a_1(t,\lambda) &= 2\zeta_\omega(t)\Big(\D^2\hat f(\lambda\tilde e_1(t))[\lambda\tilde e_2(t),\lambda\tilde e_2(t)]-\D\hat f(\lambda\tilde e_1(t))[\lambda\tilde e_1(t)]\Big) \\
&\qquad-2\,\D^2\hat f(\lambda\tilde e_1(t))[\lambda\zeta_\omega(t)\tilde e_2(t),\lambda\tilde e_2(t)]-2\,\D\hat f(\lambda\tilde e_1(t))[-\lambda\zeta_\omega(t)\tilde e_1(t)+\lambda\rho_\omega(t)\tilde e_3(t)] \\
&= -2\rho_\omega(t)\,\D\hat f(\lambda\tilde e_1(t))[\lambda\tilde e_3(t)].
\end{align*}

Finally, we evaluate the first three terms in the definition of $a_0$ in \autoref{eq:inf-cl-a0} with \autoref{eq:Dm} to get
\begin{equation}\label{eq:a0_half_simplified}
\begin{split}
a_0(t,\lambda) &= (\zeta_\omega^2(t)-\zeta_\omega(t)\varphi_\omega'(t))\,\D^3\hat f(\lambda\tilde e_1(t))[\lambda\tilde e_2(t),\lambda\tilde e_2(t),\lambda\tilde e_2(t)] \\
&\qquad+2\,\D^2\hat f(\lambda\tilde e_1(t))[\lambda\tilde e_2(t),\lambda\zeta_\omega(t)\varphi_\omega'(t)\tilde e_1(t)-\lambda\zeta_\omega'(t)\tilde e_2(t)] \\
&\qquad+2\,\D\hat f(\lambda\tilde e_1(t))[\lambda\zeta_\omega^2(t)\tilde e_2(t)+\lambda\zeta_\omega'(t)\tilde e_1(t)] \\
&\qquad-(3\zeta_\omega(t)-\varphi_\omega'(t))\,\partial_t\mathrm D^2\hat m(t,\lambda \phi_\omega(t))[\lambda\phi_\omega^\perp(t),\lambda\phi_\omega^\perp(t)] \\
&\qquad+2\,\partial_{tt}\mathrm D\hat m(t,\lambda \phi_\omega(t))[\lambda\phi_\omega^\perp(t)].
\end{split}
\end{equation}
For the second to last term herein, we write, using \autoref{eq:DtDm}, \autoref{eq:DR1}, and \autoref{eq:DR2},
\begin{align*}
\partial_t\D_k^2\hat m(t,\lambda\phi_\omega(t))[\lambda\phi_\omega^\perp(t),\lambda\phi_\omega^\perp(t)]
&= \D^3\hat f(\lambda\tilde e_1(t))[\lambda\zeta_\omega(t)\tilde e_2(t),\lambda\tilde e_2(t),\lambda\tilde e_2(t)] \\
&\qquad+2\,\D^2\hat f(\lambda\tilde e_1(t))[-\lambda\zeta_\omega(t)\tilde e_1(t)+\lambda\rho_\omega(t)\tilde e_3(t),\lambda\tilde e_2(t)].
\end{align*}
And for the last term, we find with the help of \autoref{eq:DttDm}, \autoref{eq:DR1}, \autoref{eq:DR2}, \autoref{eq:DDR1}, and \autoref{eq:DDR2} that
\begin{align*}
&\partial_{tt}\D_k\hat m(t,\lambda\phi_\omega(t))[\lambda\phi_\omega^\perp(t)] \\
&\qquad=\D^3\hat f(\lambda\tilde e_1(t))[\lambda\zeta_\omega(t)\tilde e_2(t),\lambda\zeta_\omega(t)\tilde e_2(t),\lambda\tilde e_2(t)] \\
&\qquad\qquad+\D^2\hat f(\lambda\tilde e_1(t))[-\lambda\zeta_\omega^2(t)\tilde e_1(t)+\lambda\zeta_\omega'(t)\tilde e_2(t)+\lambda\rho_\omega(t)\big(\zeta_\omega(t)-\varphi_\omega'(t)\big)\tilde e_3(t),\lambda\tilde e_2(t)] \\
&\qquad\qquad+2\D^2\hat f(\lambda\tilde e_1(t))[\lambda\zeta_\omega(t)\tilde e_2(t),-\lambda\zeta_\omega(t)\tilde e_1(t)+\lambda\rho_\omega(t)\tilde e_3(t)] \\
&\qquad\qquad+\D\hat f(\lambda\tilde e_1(t))[-\lambda\zeta_\omega'(t)\tilde e_1(t)-\lambda(\rho_\omega^2(t)+\zeta_\omega^2(t))\tilde e_2(t)+\lambda\rho_\omega'(t)\tilde e_3(t)].
\end{align*}
Inserting these two expressions into \autoref{eq:a0_half_simplified}, most of the terms cancel out and we are left with \autoref{eq:uniq-inf-cl-a0}.
\end{proof}

With these representations for the coefficients $a_0$, $a_{02}$, and $a_1$, we can check the linear independence of the equations in \autoref{eq:uniq_ThirdOrd_System} explicitly.

\begin{proposition}\label{thm:pb_uniqueness_thirdorder}
Let $f\in\mathcal{O}\setminus\mathcal{S}_{\mathrm{PB}}$ and $R\in C^2([0,T];\SO(3))$ be a rotational motion which is nondegenerate at a time $t\in[0,T]$ with associated angular velocity $\omega$ written in cylindrical coordinates $\rho_\omega$, $\phi_\omega=({\cos}\circ\varphi_\omega,{\sin}\circ\varphi_\omega)$, and $\zeta_\omega$ as defined in \autoref{eq:def-cylind}. Moreover, let $\hat m$ be the parallel-beam measurements of $f$ under $R$.

Then, \autoref{eq:uniq_ThirdOrd_System} with $a_0$, $a_{02}$, and $a_1$ defined by \autoref{eq:inf-cl-a0}, \autoref{eq:inf-cl-a02}, and \autoref{eq:inf-cl-a1} has a unique solution $(X_1,X_2)\in\R^2$ and it is given by $(X_1,X_2)=(\rho_\omega^2(t),\frac{\rho_\omega'(t)}{\rho_\omega(t)})$.
\end{proposition}

\begin{proof}
If all the vectors $(a_{02}(t,\lambda),a_1(t,\lambda))\in\R^2$, $\lambda\in\R$, were parallel, there would exist a nonzero vector $(c_{02},c_1)\in\R^2\setminus\{(0,0)\}$ orthogonal to all of them:
	\[ c_{02}a_{02}(t,\lambda)+c_1a_1(t,\lambda)=0\text{ for all }\lambda\in\R. \]
Introducing the orthonormal basis $(\tilde e_j(t))_{j=1}^3$ as in \autoref{eq:thirdOrd_ONB},
we can write this with the explicit expressions for $a_{02}$ and $a_1$ from \autoref{eq:uniq-inf-cl-a02} and \autoref{eq:uniq-inf-cl-a1} in the form
	\[\inner{\nabla \hat f(\lambda \tilde e_1(t))}{c_{02}\tilde e_2(t)-c_1\rho_\omega(t)\tilde e_3(t)}=0\text{ for all }\lambda\in\R.\]
Since $\eta\coloneqq c_{02}\tilde e_2(t)-c_1\rho_\omega(t)\tilde e_3(t)$ is a nonzero vector (as $\rho_\omega(t)\ne0$ because $R$ is nondegenerate at $t$, see \autoref{th:nondeg_D3}, and at least one of the constants $c_{02}$ and $c_1$ is nonzero) and is orthogonal to $\tilde e_1(t)$, this would contradict the fact that $f$ should be PB-asymmetric.

Therefore, there exist at least two values $\lambda_1,\lambda_2\in\R$ so that the vectors $(a_{02}(t,\lambda_1),a_1(t,\lambda_1))$ and $(a_{02}(t,\lambda_2),a_1(t,\lambda_2))$ are linearly independent, which implies that the linear equation system in \autoref{eq:uniq_ThirdOrd_System} can have at most one solution.

Moreover, we see that the vector $(X_1,X_2)\coloneqq(\rho_\omega^2(t),\frac{\rho_\omega'(t)}{\rho_\omega(t)})$ satisfies with the formulas from \autoref{th:pb-uniq_thirdOrd_coeff} for the coefficients $a_0$, $a_{02}$, and $a_1$ that
\[ a_{02}(t,\lambda)X_1+a_1(t,\lambda)X_2 = 2\rho_\omega^2(t)\inner{\nabla\hat f(\lambda\tilde e_1(t))}{\lambda\tilde e_2(t)}-2\rho_\omega'(t)\inner{\nabla\hat f(\lambda\tilde e_1(t))}{\lambda\tilde e_3(t)} = -a_0 \]
for every $\lambda\in\R$.
\end{proof}

\autoref{eq:uniq_ThirdOrd_System} thus allows us to uniquely recover for a PB-asymmetric object under a nondegenerate rotation $R\in C^2([0,T];\SO(3))$ from the knowledge of the cylindrical components $\phi_\omega$ and $\zeta_\omega$ of the corresponding angular velocity $\omega$ the values $\rho_\omega^2(t)$ and $\frac{\rho_\omega'}{\rho_\omega}(t)$ at some time $t\in[0,T]$. However, this only allows us to retrieve the absolute value $\abs{\rho_\omega(t)}$ of the remaining cylindrical component $\rho_\omega(t)$.

But this restriction that $\rho_\omega(t)$ can only be determined up to its sign is just a manifestation of the fact that we cannot obtain the rotation $R$, but at best the equivalence class $\{R,\Sigma R\Sigma\}$, $\Sigma\coloneqq\operatorname{diag}(1,1,-1)$, as we discussed before formulating \autoref{pr:pb}. Indeed, if $\omega$ is the angular velocity associated to $R$, then we get with the identity $(Ax)\times(A\tilde x) = \det(A)A(x\times\tilde x)$ for all $x,\tilde x\in\R^3$ and $A\in\mathrm O(3)$ that
\begin{equation}\label{eq:equivalence_class_angular_velocity}
\Sigma R'(t)\Sigma x = \Sigma R(t)(\omega(t)\times(\Sigma x)) = -\Sigma R(t)\Sigma(\Sigma\omega(t)\times x)\text{ for all }t\in[0,T],\;x\in\R^3,
\end{equation}
meaning that $-\Sigma\omega(t)=(-\omega_1(t),-\omega_2(t),\omega_3(t))$ is the angular velocity associated to $\Sigma R\Sigma$.

Thus, putting \autoref{thm:pb_uniqueness_firstorder} and \autoref{thm:pb_uniqueness_thirdorder} together, we get the reconstruction of the rotational motion in the nondegenerate case.

\begin{theorem}\label{thm:pb_uniqueness}
	Let $f^{(1)},f^{(2)}\in\mathcal{O}\setminus\mathcal S_{\mathrm{PB}}$ be two PB-asymmetric objects, $R^{(1)},R^{(2)}\in C^2([0,T];\SO(3))$ be two rotational motions which are normalized by $R^{(1)}(0)=R^{(2)}(0)=\I_{3\times3}$ and nondegenerate at every point $t\in[0,T]$, and $\omega^{(1)}$ and $\omega^{(2)}$ be their respective associated angular velocities. Moreover, let $\hat{m}^{(j)}$ denote the parallel-beam measurements of $f^{(j)}$ under the rotation $R^{(j)}$, $j\in\{1,2\}$.
	
	We then have that $\hat m^{(1)}=\hat m^{(2)}$ if and only if
\begin{alignat*}{3}
&R^{(2)}=R^{(1)}&&\text{ and }\hat f^{(1)}(\kappa)=\hat f^{(2)}(\kappa)&&\text{ for all }\kappa\in K\text{ or} \\
&R^{(2)}=\Sigma R^{(1)}\Sigma&&\text{ and }\hat f^{(1)}(\kappa)=\hat f^{(2)}(\Sigma\kappa)&&\text{ for all }\kappa\in K,
\end{alignat*}
where $K\coloneqq\set*{R^{(1)}(t)\begin{pmatrix}k\\0\end{pmatrix}}{t\in[0,T],\,k\in\R^2}$.
\end{theorem}

\begin{proof}
If the object and the rotation coincide, the measurements clearly are identical; and if $R^{(2)}=\Sigma R^{(1)}\Sigma$, we have seen in \autoref{eq:equivalence_class} that the measurements of $f^{(1)}$ under $R^{(1)}$ are the same as those of $f^{(2)}\circ\Sigma$ under $R^{(2)}$.

On the other hand, if we have that $\hat m^{(1)}=\hat m^{(2)}$, we write for $j\in\{1,2\}$ the angular velocity $\omega^{(j)}$ in the cylindrical coordinates $\rho_{\omega^{(j)}}$, $\phi_{\omega^{(j)}}=({\cos}\circ\varphi_{\omega^{(j)}},{\sin}\circ\varphi_{\omega^{(j)}})$, and $\zeta_{\omega^{(j)}}$ and remark that the nondegeneracy assumptions imply according to \autoref{th:nondeg_D3} that $\rho_{\omega^{(j)}}$ does not vanish anywhere. Then we get from \autoref{thm:pb_uniqueness_firstorder} for every $t\in[0,T]$ that $(\phi_{\omega^{(1)}}(t),\zeta_{\omega^{(1)}}(t))=(\phi_{\omega^{(2)}}(t),\zeta_{\omega^{(2)}}(t))$ as both are the unique solution of the equation
\[ \partial_t\hat m(t,\lambda\phi)=\zeta \inner{\nabla_k \hat m(t,\lambda\phi)}{\lambda\phi^\perp}\text{ for all }\lambda\in\R \]
for $(\phi,\zeta)\in S^1_+\times\R$, $S^1_+\coloneqq\set{(\cos(\alpha),\sin(\alpha))}{\alpha\in[0,\pi)}$.

We then define the coefficients $a_0$, $a_{02}$, and $a_1$ by \autoref{eq:inf-cl-a0}, \autoref{eq:inf-cl-a02}, and \autoref{eq:inf-cl-a1} with $\phi_\omega\coloneqq\phi_{\omega^{(1)}}$, $\varphi_\omega\coloneqq\varphi_{\omega^{(1)}}$, and $\zeta_\omega=\zeta_{\omega^{(1)}}$. \autoref{eq:uniq_ThirdOrd_System} for the variables $X_1$ and $X_2$ at an arbitrary time $t\in[0,T]$ then has, according to \autoref{thm:pb_uniqueness_thirdorder}, the unique solution
\[ (X_1,X_2)=\left(\rho_{\omega^{(1)}}^2(t),\frac{\rho_{\omega^{(1)}}'(t)}{\rho_{\omega^{(1)}}(t)}\right)=\left(\rho_{\omega^{(2)}}^2(t),\frac{\rho_{\omega^{(2)}}'(t)}{\rho_{\omega^{(2)}}(t)}\right). \]
We thus have $\abs{\rho_{\omega^{(1)}}(t)}=\abs{\rho_{\omega^{(2)}}(t)}$, which means that $\omega^{(1)}(t)=\omega^{(2)}(t)$ or $\omega^{(1)}(t)=-\Sigma\omega^{(2)}(t)$, where $\Sigma\coloneqq\diag(1,1,-1)$.

Since $\omega^{(1)}$ and $\omega^{(2)}$ are continuous and their first two components are everywhere nonzero, we find that we either have $\omega^{(1)}(t)=\omega^{(2)}(t)$ for every $t\in[0,T]$ or $\omega^{(1)}(t)=-\Sigma\omega^{(2)}(t)$ for every $t\in[0,T]$. Since $R^{(j)}$ is uniquely determined by $\omega^{(j)}$ via \autoref{eq:def-omega} and the initial condition $R^{(j)}(0)=\I_{3\times3}$, $j\in\{1,2\}$, we get that $R^{(1)}=R^{(2)}$ in the first case and, with \autoref{eq:equivalence_class_angular_velocity}, that $R^{(1)}=\Sigma R^{(2)}\Sigma$ in the second case.

The equality of the measurements then enforces that we have for all $t\in[0,T]$ and $k\in\R^2$
\[ \hat f^{(1)}\left(R^{(1)}(t)\begin{pmatrix}k\\0\end{pmatrix}\right) = \hat m(t,k) = \hat f^{(2)}\left(R^{(1)}(t)\begin{pmatrix}k\\0\end{pmatrix}\right) \]
in the first and
\[ \hat f^{(1)}\left(R^{(1)}(t)\begin{pmatrix}k\\0\end{pmatrix}\right) = \hat m(t,k) = \hat f^{(2)}\left(\Sigma R^{(1)}(t)\Sigma\begin{pmatrix}k\\0\end{pmatrix}\right) = \hat f^{(2)}\left(\Sigma R^{(1)}(t)\begin{pmatrix}k\\0\end{pmatrix}\right) \]
in the second case.
\end{proof}

\subsection{PB-asymmetric point sets and construction of PB-asymmetric functions}
To study the set of PB-asymmetric objects, we will, as in the case of diffraction tomography data, first consider objects which consist of finitely many point masses so that the condition from \autoref{def:pb-symmetry} reduces to a constraint on the arrangement of their centers.

\begin{definition}\label{def:pb_asym_points}
	We call a finite set $P\subset\R^3$ a PB-asymmetric point set if we can find for every direction $\xi\in S^2$ two points $p_1(\xi),p_2(\xi)\in P$ such that the vectors $\xi$, $p_1(\xi)$, and $p_2(\xi)$ are linearly independent and
	\begin{equation} \label{eq:pb-asym-set1}
		\inner{p_j(\xi)}\xi\ne\inner p\xi\text{ for all }p\in P\setminus\{p_j(\xi)\},\;j\in\{1,2\}.
	\end{equation}
\end{definition}

To guarantee the existence of two points from a set $P$ satisfying the condition given in \autoref{eq:pb-asym-set1}, it must contain at least six points, since we can find for every set of four points $p_j \in \R^3$, $j \in \{1, 2, 3, 4\}$, a direction $\xi\in S^2$ orthogonal to the vectors $p_1 - p_2$ and $p_3 - p_4$ so that $\inner{p_1}\xi=\inner{p_2}\xi$ and $\inner{p_3}\xi=\inner{p_4}\xi$.

However, even if we have two more points $p_5$ and $p_6$ in $P$ with different components in the direction $\xi$, the vectors $\xi$, $p_5$, and $p_6$ could still be linearly dependent so that $P$ would nevertheless fail to be a PB-asymmetric point set.
To exclude this, we will require that $P$ has at least seven points, placed in such a way that there are for every direction $\xi$ at most four points sharing the same projection onto $\xi$ with some other point of $P$ and that three remaining elements in $P$ are linearly independent.

\begin{lemma}\label{thm:pb_asym_points_existence}
	Let $N\geq7$ and $P=\{p_j\}_{j=1}^N\subset\R^3$ be a finite set of points with the following properties:
	\begin{enumerate}
		\item $\det(p_i,p_j,p_k)\neq 0$ for all distinct $i,j,k\in\{1,\ldots,N\}$ and\label{eq:pb_point_exist_cond3}
		\item $\det(p_{i_1}-p_{j_1},p_{i_2}-p_{j_2},p_{i_3}-p_{j_3})\neq0$ for all distinct index pairs $(i_\ell,j_\ell)\in\{(i, j) \mid 1 \leq i < j \leq N\}$, $\ell\in\{1,2,3\}$, which contain at least four different indices: $\abs{\{i_1, j_1, i_2, j_2, i_3, j_3\}} \geq 4$. \label{eq:pb_point_exist_cond2}
	\end{enumerate}
	Then, $P$ is a PB-asymmetric point set. 
\end{lemma}

\begin{proof}
	Let $\xi\in S^2$ be an arbitrary vector and define the set $C$ of components of the points in $P$ in the direction $\xi$:
	\[ C\coloneqq\set{\inner p\xi}{p\in P},\]
	and the set $P_c$ of all points with the same projection $c\in C$ onto $\xi$:
	\[P_c\coloneqq\set{p\in P}{\inner p\xi=c}. \]

	We now claim that there always exist at least three values $c\in C$ with $\abs{P_c}=1$.
	\begin{itemize}
		\item
		If there was a value $c\in C$ with $P_c\supset\{\tilde p_j\}_{j=1}^4$ with four different elements $\tilde p_j$, $j\in\{1,2,3,4\}$, then the determinant of the three vectors $\tilde p_1-\tilde p_2$, $\tilde p_2-\tilde p_3$, and $\tilde p_3-\tilde p_4$ would be zero as they would be all perpendicular to $\xi$, which would contradict the assumption in \autoref{eq:pb_point_exist_cond2}.
		\item
		If we have a $c\in C$ with $\abs{P_c}=3$, then we get with $P_c=\{\tilde p_1,\tilde p_2,\tilde p_3\}$ that $\inner{\tilde p_2-\tilde p_3}\xi = \inner{\tilde p_1-\tilde p_2}\xi = 0$. If we assume that $c'\in C$ was another element with $P_{c'}\supset\{\tilde p_4,\tilde p_5\}$, then we would additionally have $\inner{\tilde p_4-\tilde p_5}\xi = 0$.
		But, since the vectors $\tilde p_3-\tilde p_2$, $\tilde p_2-\tilde p_1$, and $\tilde p_4-\tilde p_5$ are linearly independent by the condition given in \autoref{eq:pb_point_exist_cond2}, this would imply $\xi=0$, which is impossible.

		So if we have one value $c\in C$ with $\abs{P_c}=3$, we necessarily get that $\abs{P_{c'}}=1$ for all $c'\ne c$ which implies that there are at least four values $c$ with $\abs{P_c}=1$.
		\item
		Assume that we have two different values $c,c'\in C$ such that $P_c$ and $P_{c'}$ both have two elements: $P_c=\{\tilde p_1,\tilde p_2\}$, $P_{c'}=\{\tilde p_3,\tilde p_4\}$.
		If there was another value $c''\in C\setminus\{c,c'\}$ for which $P_{c''}$ has at least two elements, $P_{c''}\supset \{\tilde p_5,\tilde p_6\}$, then we would have $\inner{\tilde p_1-\tilde p_2}\xi = \inner{\tilde p_3-\tilde p_4}\xi = \inner{\tilde p_5-\tilde p_6}\xi = 0$, which would again imply $\xi=0$. So there must be at least three values $\tilde c\in C$ with $\abs{P_{\tilde c}}=1$.
	\end{itemize}

	Hence, we can find for every direction $\xi$ a set $\tilde P_\xi\subset P$ consisting of at least three points having distinct projections onto $\xi$ with respect to all other points: $\inner{\tilde p}\xi\ne\inner p\xi$ for all $\tilde p\in\tilde P_\xi$ and $p\in P\setminus\{\tilde p\}$.

	Because of the assumption in \autoref{eq:pb_point_exist_cond3} that every set of three points in $P$ is linearly independent and $\tilde P_\xi$ contains at least three points, we can pick two points $p_1(\xi),p_2(\xi)\in\tilde P_\xi$ such that $p_1(\xi)$, $p_2(\xi)$, and $\xi$ are linearly independent. By construction of the set $\tilde P_\xi$, they then also fulfill \autoref{eq:pb-asym-set1}.
\end{proof}

By taking an appropriately weighted sum of radially symmetric functions centered at points of a PB-asymmetric point set $P\coloneqq\{p_j\}_{j=1}^N$ such that the first moments of this sum vanishes, we can construct a PB-asymmetric function. For this, the nonzero weights have to be chosen such that $\sum_{j=1}^Nw_jp_j=0$, and we can guarantee the existence of such weights with \autoref{thm:dt_weights} if we construct the PB-asymmetric point set as in \autoref{thm:pb_asym_points_existence} so that any three distinct points from $P$ are linearly independent.

\begin{lemma}\label{thm:pb_def_dist}
	Let $P=\{p_j\}_{j=1}^N\subset \B_{r_{\mathrm s}}^3$ be a PB-asymmetric point set and $\{w_j\}_{j=1}^N\subset\R\setminus\{0\}$ be corresponding weights with $\sum_{j=1}^N w_j p_j=0$. Moreover, let $\psi\in L^2(\B_{r_{\mathrm s}}^3)\setminus\{0\}$ be a radial function with $\supp(\psi)\subset \B^3_\delta$, where $\delta\coloneqq\mathrm{dist}(P,\partial \B_{r_{\mathrm s}}^3)$.

The function
	\begin{equation}\label{eq:pb_def_dist}
		\Psi\colon\B_{r_{\mathrm s}}^3\to\C,\;\Psi(x)\coloneqq\sum_{j=1}^{N}w_j \psi(x-p_j),
	\end{equation}
	in $\mathcal O$ is then a PB-asymmetric function.
\end{lemma}
\begin{proof}
	We have already seen in the proof of \autoref{thm:dt_def_dist} that such a function $\Psi$ is contained in the set of admissible functions $\mathcal{O}$. 
		
	Assume by contradiction that $\Psi$ was PB-symmetric, that is, there exist directions $\xi,\eta\in S^2$ with $\inner\xi\eta=0$ such that
	\begin{equation}\label{eq:asum_Psi_PBsym}
		\inner{\nabla \hat \Psi(\lambda\xi)}{\eta}=0 \text{ for all }\lambda\in\R.
	\end{equation}
Plugging in the gradient of the Fourier transform of $\Psi$ from \autoref{eq:dt_psi_fourier_grad} into this expression, and proceeding analogously to the proof of \autoref{thm:dt_def_dist}, we obtain
\[ \hat\psi(\lambda\xi)\sum_{j=1}^N w_j \e^{-\i\lambda\inner{p_j}{\xi}}\inner{p_j}{\eta}=0. \]
	Since $\psi$ is radial and has compact support, its Fourier transform $\hat\psi$ is radial and analytic and there consequently exists a nonempty interval $I \subset \R$ such that $\hat\psi(\lambda\xi) \neq 0$ for all $\lambda \in I$ because the analyticity would otherwise imply that $\hat\psi$ vanishes everywhere. Therefore, it would have to hold that
\begin{equation}\label{eq:asum_Psi_PBsym_consequence}
\sum_{j=1}^N w_j \e^{-\i\lambda\inner{p_j}{\xi}}\inner{p_j}{\eta}=0\text{ for all }\lambda\in I.
\end{equation}

	Using the properties of the PB-asymmetric point set, we now choose two points $p_1(\xi)$ and $p_2(\xi)$ from $P$ such that $\xi$, $p_1(\xi)$, and $p_2(\xi)$ are linearly independent and we have \autoref{eq:pb-asym-set1}. To keep the notation simple, we want to assume that the set $P$ is enumerated in a way that $p_1(\xi)=p_1$ and $p_2(\xi)=p_2$.

Since the complex exponential functions $\lambda\mapsto \e^{\i \lambda c_j}$, are linearly independent on $I$ for distinct parameters $c_j\in\R$, $j\in\{1,\ldots,N\}$, it then follows that the functions $\e^{-\i\lambda\inner{p_1}{\xi}}$, $\e^{-\i\lambda\inner{p_2}{\xi}}$ and $\sum_{j=3}^Nw_j\e^{-\i\lambda\inner{p_j}{\xi}}$ are linearly independent so that we would get $w_1\inner{p_1}{\eta}=w_2\inner{p_2}{\eta}=0$ from \autoref{eq:asum_Psi_PBsym_consequence}. However, this would imply that $\xi$, $p_1$, and $p_2$ are all in the subspace orthogonal to $\eta$ and therefore, in contradiction to our assumption, not linearly independent.

Hence, \autoref{eq:asum_Psi_PBsym} cannot be fulfilled and $\Psi$ must be PB-asymmetric.
\end{proof}

\subsection{Genericity of PB-asymmetric functions}
As in the diffraction tomography problem, we want to conclude the discussion by ensuring that a unique reconstruction is achievable in the generic case. To this end, we will prove that the set of PB-symmetric functions $\mathcal{S}_{\mathrm{PB}}$ is closed and nowhere dense (see \autoref{def:generic}) in the set $\mathcal{O}$ of admissible functions so that its complement, the set of PB-asymmetric functions, is generic. 

\begin{lemma}\label{thm:pb_S_closed}
	The set $\mathcal{S}_{\mathrm{PB}}$ is a closed subset of $\mathcal{O}$.
\end{lemma}
\begin{proof}
	The proof goes along the lines of the proof of \autoref{thm:dt_S_closed} for the closure of $\mathcal{S}_{\mathrm{DT}}$. 
	Let $(f_j)_{j=1}^\infty\subset \mathcal{S}_{\mathrm{PB}}$ be a converging sequence in $L^2(\B^3_{r_{\mathrm s}})$ with limit $f \in \mathcal{O}$. By the definition of $\mathcal{S}_{\mathrm{PB}}$, there exist for each $f_j$ directions $\xi_j, \eta_j \in S^2$ with $\inner{\xi_j}{\eta_j}=0$ such that
		\[ \inner{\nabla \hat f_j(\lambda\xi_j)}{\eta_j}= 0\text{ for all }\lambda\in\R. \]
	
	Since $S^2$ is compact, we can find converging subsequences $(\xi_{j_\ell})_{\ell=1}^\infty$ and $(\eta_{j_\ell})_{\ell=1}^\infty$ with limits $\xi, \eta \in S^2$ satisfying $\inner{\xi}{\eta}= 0$. By the uniform convergence of $\nabla \hat{f_j}$ to $\nabla \hat{f}$, it then follows that
	\[ \inner{\nabla \hat f(\lambda\xi)}{\eta}= 0\text{ for all }\lambda\in\R. \]
\end{proof}

Since being nowhere dense is for the closed set $\mathcal S_{\mathrm{PB}}$ equivalent to having an empty interior, we now proceed to demonstrate, analogously to \autoref{thm:dt_S_nowheredense}, that every point in $\mathcal{S}_{\mathrm{PB}}$ is a boundary point. We do so by showing that we can find for every PB-symmetric function $f\in\mathcal S_{\mathrm{PB}}$ an arbitrarily small PB-asymmetric perturbation $\Psi$ as in \autoref{eq:pb_def_dist} such that $f+\Psi$ is PB-asymmetric.

\begin{theorem}\label{thm:pb_S_nowheredense}
	The set $\mathcal{S}_{\mathrm{PB}}$ is nowhere dense in $\mathcal{O}$.
\end{theorem}
\begin{proof}
	Let us assume by contradiction that there was a PB-symmetric function $f\in\mathcal{S}_{\mathrm{PB}}$ in the interior of $\mathcal S_{\mathrm{PB}}$. We could thus find an open ball $U_\delta(f)\subset\mathcal O$ with some radius $\delta>0$ around $f$ which is contained in $\mathcal{S}_{\mathrm{PB}}$.
By \autoref{thm:dt_smooth_approx}, we could then find a parameter $\varepsilon>0$ and a function $g\in C_{\mathrm c}^\infty(\B^3_{r_{\mathrm s}})\cap U_\delta(f)$ with $\supp(g)\subset \B^3_{r_{\mathrm s}-2\varepsilon}$. Let $\tilde\delta>0$ be selected such that 
	\begin{equation}\label{eq:pb_nested_interior_point}
		U_{\tilde{\delta}}(g)\subset U_\delta(f)\subset\mathcal{S}_{\mathrm{PB}}.
	\end{equation}
	
	We now pick a PB-asymmetric set $P\coloneqq\{p_j\}_{j=1}^N$ for which we have corresponding weights $\{w_j\}_{j=1}^N\subset\R\setminus\{0\}$ with $\sum_{j=1}^Nw_jp_j=0$ and consider a function $\Psi_{\tilde\varepsilon}$ with the same structure as in \autoref{eq:pb_def_dist}, where we choose for some $\tilde\varepsilon<\varepsilon$ for $\psi$ the indicator function $\psi\coloneqq\bm{1}_{\B^3_{\tilde\varepsilon}}$ of the ball with radius $\tilde\varepsilon$:
	\begin{equation*}
		\Psi_{\tilde\varepsilon}(x)\coloneqq\sum_{j=1}^{N}w_j \bm{1}_{\B^3_{\tilde\varepsilon}}(x-p_j).
	\end{equation*}
	According to \autoref{thm:dt_def_dist}, $\Psi_{\tilde\varepsilon}\in\mathcal O$ is PB-asymmetric. Moreover, we can select $\tilde\varepsilon$ so small that $\norm{\Psi_{\tilde\varepsilon}}_{L^2}\leq\tilde\delta$.
	
	We now prove that $g+\Psi_{\tilde\varepsilon}$ is also PB-asymmetric in order to get a contradiction to \autoref{eq:pb_nested_interior_point}. 
	Assume that $g+\Psi_{\tilde\varepsilon}$ was PB-symmetric so that we had for some pair of directions $\xi,\eta\in S^2$ with $\inner{\xi}{\eta}=0$ that
		\[\inner{\nabla (\hat g+\hat\Psi_{\tilde\varepsilon})(\lambda\xi)}{\eta}=0\text{ for all }\lambda\in\R.\]
	In particular, it had to hold that
	\begin{equation}\label{eq:pb_nowheredense_contradiction}
	\abs{\inner{\nabla \hat g(\lambda\xi)}{\eta}}= \abs{\inner{\nabla \hat\Psi_{\tilde\varepsilon}(\lambda\xi)}{\eta}}\text{ for all }\lambda\in\R.
	\end{equation}
	Since $\nabla\hat g$ is the Fourier transform of the smooth and compactly supported function $\tilde g\in C_{\mathrm c}^\infty(\B_{r_{\mathrm s}}^3)$, $\tilde g(x)\coloneqq-\i xg(x)$, we know from the Paley--Wiener theorem (see \cite[Theorem 7.3.1]{Hoe03}) that there exists a real constant $C>0$ such that
	\begin{equation}\label{eq:pb_paley_wiener}
		\abs{\inner{\nabla \hat g(\lambda\xi)}{\eta}}\leq C(1+\abs{\lambda})^{-3} \text{ for all }\lambda\in\R.
	\end{equation}
	Moreover, the fact that the indicator function $\psi$ is radial implies because of $\inner{\xi}{\eta}=0$ that
	\[\abs{\inner{\nabla \hat\Psi_{\tilde\varepsilon}(\lambda\xi)}{\eta}}=\abs*{\hat\psi(\lambda\xi)\sum_{j=1}^N w_j \e^{-\i\lambda\inner{p_j}{\xi}}\inner{p_j}{\eta}},\]
	where the Fourier transform $\hat\psi$ of the indicator function of $\B_{\tilde\varepsilon}^3$ is explicitly given by
	\[\hat\psi(\kappa)=\sqrt{\frac{2}{\pi}}\frac{\sin(\tilde\varepsilon\norm\kappa)-\tilde\varepsilon\norm\kappa\cos(\tilde\varepsilon\norm\kappa)}{\norm\kappa^3}\text{ for all }\kappa\in\R^3\setminus\{0\}.\]
	With the estimate from \autoref{eq:pb_paley_wiener}, \autoref{eq:pb_nowheredense_contradiction} would then give us the inequality
	\[\abs*{\sqrt{\frac{2}{\pi}}\left(\frac{\sin(\tilde\varepsilon\abs{\lambda})}{\abs{\lambda}}-\tilde\varepsilon\cos(\tilde\varepsilon\abs{\lambda})\right)\sum_{j=1}^N w_j \e^{-\i\lambda\inner{p_j}{\xi}}\inner{p_j}{\eta}}\leq C\frac{\abs{\lambda}^{2}}{(1+\abs{\lambda})^{3}} \text{ for all } \lambda\in\R.\] 

	Taking the limit $\lambda\to\infty$ herein, both the right-hand side and $\frac{\sin(\tilde\varepsilon\abs{\lambda})}{\abs{\lambda}}$ converge to zero.  According to \autoref{thm:pb_def_dist}, $\sum_{j=1}^N w_j \e^{-\i\lambda\inner{p_j}{\xi}}\inner{p_j}{\eta}$ does not vanish identically and by \autoref{thm:exp-almost-period} and \autoref{thm:product_almost_periodic}, its product with a periodic function is an almost periodic function. Thus, we arrive at a contradiction, as the right-hand side becomes arbitrarily small as $\lambda$ increases, while the left-hand side, being almost periodic, attains values away from zero for arbitrarily large $\lambda$.
\end{proof}

\section{Conclusion}
We have investigated the uniqueness of rotational motion reconstruction for trapped biological samples undergoing unknown, time-dependent rotations within the frameworks of two approximate models: diffraction tomography under the Born approximation and parallel-beam tomography. For diffraction tomography, we introduced the concept of DT-asymmetry, which guarantees unique recovery of all rotational parameters using the infinitesimal common circle method. For parallel-beam tomography, we derived the necessary conditions on the motion that enable recovery of the rotation via the infinitesimal common line method and defined the property of PB-asymmetry, which enables the unique determination of all motion parameters up to an orthogonal transformation.
Our results show that, while non-uniqueness may arise for certain symmetric objects for each model, the set of such objects is nowhere dense within the space of admissible samples. This indicates that, for generic objects, reconstructing rotational motion from measurement data is achievable.

These findings provide a theoretical foundation for the use of infinitesimal motion reconstruction methods as a preprocessing step in three-dimensional tomographic imaging of optically or acoustically trapped particles. By ensuring that the recovered motion is unique in practically all relevant scenarios, our results support the reliability of these methods in experimental settings where the sample’s orientation is not precisely controlled. Our definitions of PB- and DT-asymmetry could also assist in the design of 3D objects for simulations, ensuring that infinitesimal methods remain robust against failures caused by unexpected symmetries in the artificial sample.

\appendix
\section{Almost periodic functions}
The concept of almost periodic functions was introduced by Harald Bohr in the 1920s as a generalization of periodic functions \cite{Boh25}. They are characterized by the property that, for any desired level of accuracy, there exist almost periods after which the function nearly repeats its values. 
\begin{definition} \label{def:almost-period}
	A function $f\colon\R\to\C$ is called almost periodic if there exists for every $\delta>0$ a parameter $\ell_\delta>0$ such that every interval $I$ of length $\ell_\delta$ contains a value $\tau\in I$ with the property
	\[ |f(x+\tau)-f(x)|<\delta \text{ for all }x\in\R. \]
	The number $\tau$ is called the $\delta$-almost period of $f$.
\end{definition}

Clearly, every periodic function is almost periodic, and it is shown in \cite[Satz III and Satz IV]{Boh25} that the sum and the product of two almost periodic functions is again almost periodic.

\begin{proposition}\label{thm:product_almost_periodic}
Let $f_1,f_2\colon\R\to\C$ be almost periodic functions. Then, $f_1+f_2$ and $f_1f_2$ are almost periodic, too.
\end{proposition}

In particular, this implies that all trigonometric polynomials are almost periodic functions.

\begin{corollary}\label{thm:exp-almost-period}
Every function $f\colon\R\to\C$ of the form 
	\[ f(x)\coloneqq\sum_{j=1}^N a_j \e^{\i b_j x}\]
with coefficients $a_j\in\C$ and $b_j\in\R$, $j\in\{1,\dots,N\}$, $N\in\N$, is almost periodic.
\end{corollary}

In fact, the space of almost periodic functions is exactly the closure of the set of these complex trigonometric polynomials under the supremum norm, see \cite[Theorem I, p.226]{BesBoh31}. 

This almost periodicity allows us to obtain the linear independence of a certain class of exponential functions.

\begin{lemma}\label{thm:dt_compl_exp_linindep}
Let $N\in\N$, $k_0\in(0,\infty)$, $h$ be the map given by \autoref{eq:dt_def_h}, $I\subset(-k_0,k_0)$ be an arbitrary nonempty interval, and $(a_j,b_j)$ be distinct points in $\R^2$ for $j\in\{1,\ldots,N\}$.
	
The functions
\[ \left(\mu\mapsto\mu\e^{\i(a_j \mu+b_jh(\mu))}\right)_{j=1}^N,\;\left(\mu\mapsto h(\mu)\e^{\i(a_j \mu+b_jh(\mu))}\right)_{j=1}^N \]
on $I$ are then linearly independent.
\end{lemma}
\begin{proof}
We assume by contradiction that the functions were linearly dependent. Then, there would exist complex coefficients $c_j,d_j\in\C$ such that
\[F(\mu)\coloneqq\sum_{j=1}^Nc_j \mu \e^{\i(a_j \mu+b_jh(\mu))} + \sum_{j=1}^N d_j h(\mu) \e^{\i(a_j \mu+b_jh(\mu))}= 0\text{ for all }\mu\in I,\]
where we presume that $(c_j,d_j)\ne(0,0)$ for every $j\in\{1,\ldots,N\}$ by otherwise dropping these terms from the sum.

The function $h$ can be analytically extended to the complex plane, except along a branch cut connecting $-k_0$ and $k_0$, which we explicitly choose as $\Gamma \coloneqq \set{z \in \C}{\Im(z) > 0,\,\abs z = k_0}$. Consequently, $F$ also extends holomorphically to $\C \setminus \Gamma$ so that we would have
\[F(z)=0 \text{ for all }z\in\C\setminus\Gamma.\]

For $\mu \in (k_0, \infty)$, we then get the extension $h(\mu) = -\i\sqrt{\mu^2 - k_0^2}-k_0$ and we would thus get
\[\sum_{j=1}^N(c_j \mu+d_j h(\mu)) \e^{\i(a_j \mu-k_0 b_j)}\e^{b_j\sqrt{\mu^2-k_0^2}} = 0\text{ for all }\mu\in(k_0,\infty).\]
Assuming that the sequence $(b_j)_{j=1}^N$ is ordered decreasingly and letting $n\in\{1,\ldots,N\}$ be the index with $b_1=b_n>b_j$ for all $j>n$, we write this equation in the form
\begin{equation}\label{eq:linear_dependence_condition}
\begin{split}
&\e^{-\i k_0 b_1}\sum_{j=1}^n(c_j-\i d_j) \e^{\i a_j \mu} + \left(\frac{h(\mu)}\mu+\i\right)\e^{-\i k_0 b_1}\sum_{j=1}^nd_j\e^{\i a_j \mu} \\ 
&\qquad+\sum_{j=n+1}^N\left(c_j+d_j\frac{h(\mu)}\mu\right) \e^{\i(a_j \mu-k_0 b_j)}\e^{(b_j-b_1)\sqrt{\mu^2-k_0^2}}=0\text{ for all }\mu\in(k_0,\infty).
\end{split}
\end{equation}
The first term herein is, according to \autoref{thm:exp-almost-period}, almost periodic, the second term converges because of $\lim_{\mu\to\infty}(h(\mu)+\i\mu)=-k_0$ to zero in the limit $\mu\to\infty$, and so does the last term as $b_j - b_1 < 0$ for all $j > n$. The almost periodic first term would thus have to tend to zero at infinity, too, which would only be possible if it is constantly equal to zero. We would therefore arrive at the condition
\[ \sum_{j=1}^n(c_j-\i d_j) \e^{\i a_j z}=0\text{ for all }z\in\C, \]
which would imply that $c_j=\i d_j$ for all $j\in\{1,\ldots,n\}$.

So the first term in \autoref{eq:linear_dependence_condition} would disappear and the equation would reduce to
\[ \e^{-\i k_0 b_1}\sum_{j=1}^nd_j\e^{\i a_j \mu}+\sum_{j=n+1}^N\frac{c_j\mu+d_jh(\mu)}{h(\mu)+\i\mu}\e^{\i(a_j \mu-k_0 b_j)}\e^{(b_j-b_1)\sqrt{\mu^2-k_0^2}}=0\text{ for all }\mu\in(k_0,\infty). \]
As before, the first term is an almost periodic function whereas the second term still converges to zero for $\mu\to\infty$ so that we would require
\[ \sum_{j=1}^nd_j \e^{\i a_j z}=0\text{ for all }z\in\C, \]
which would only be fulfilled if $d_j=0$ for all $j\in\{1,\ldots,n\}$. Thus, we would have $c_j=d_j=0$ for all $j\in\{1,\ldots,n\}$, which contradicts our assumption.
\end{proof}

\section{Technical results for the proof of \autoref{thm:dt_S_nowheredense}}

The first step in the proof of \autoref{thm:dt_S_nowheredense} uses that if $f\in\mathcal O$ was an inner point of the set $\mathcal S_{\mathrm{DT}}$ of DT-symmetric functions, then there would exist a smoothed version $g\in C_{\mathrm c}^\infty(\B^3_{r_{\mathrm s}})\cap\mathcal O$ of $f$ which is also an inner point of $\mathcal S_{\mathrm{DT}}$.

\begin{lemma}\label{thm:dt_smooth_approx}
	Let $f\in\mathcal{O}$. Then for every $\delta>0$, there exists an $\varepsilon>0$ and a function $g\in C_{\mathrm c}^\infty(\B^3_{r_{\mathrm s}})\cap\mathcal{O}$ with $\supp(g)\subset \B^3_{r_{\mathrm s}-2\varepsilon}$ and $\norm{f-g}_{L^2}<\delta$.
\end{lemma}
\begin{proof}
	Let $\delta>0$. We choose a smooth cutoff function $\chi_\varepsilon\in C_{\mathrm c}^\infty(\B^3_{r_{\mathrm s}})$ with $\supp(\chi_\varepsilon)\subset \B^3_{r_{\mathrm s}-3\varepsilon}$ and $\chi_\varepsilon(x)\in[0,1]$ for all $x\in\B^3_{r_{\mathrm s}}$ such that $\chi_\varepsilon(x)=1$ for $\abs{x}<r_{\mathrm s}-4\varepsilon$ and truncate the function $f$ to $f_\chi\coloneqq f\chi_\varepsilon$, which then fulfills $\supp(f_\chi)\subset \B^3_{r_{\mathrm s}-3\varepsilon}$. Hereby, we choose $\varepsilon\in(0,\frac18r_{\mathrm s})$ sufficiently small such that
	\begin{equation}\label{eq:ineq1}
		\norm{f-f_\chi}_{L^2}\leq \int_{\B^3_{r_{\mathrm s}}\setminus\B^3_{r_{\mathrm s}-4\varepsilon}}\abs{f(x)}^2\d x \le \frac{\delta}{3C},
	\end{equation}
where we will specify the constant $C\ge1$ later in \autoref{eq:constant_smooth_approx}.

	However, $f_\chi$ is not guaranteed to have vanishing first order moments, so it might not belong to $\mathcal{O}$. To address this issue, we introduce a corrective function. Specifically, we select a radial function $\phi\in C_{\mathrm c}(\B^3_{r_{\mathrm s}})\setminus\{0\}$ with $\supp(\phi)\subset\B^3_r$ for some $r\in(0,\frac12r_{\mathrm s})$ and $\phi(x)\in[0,1]$ for all $x\in\B^3_{r_{\mathrm s}}$ and define
	\[\tilde f_\chi(x)\coloneqq f_\chi(x)-\sum_{j=1}^{3}a_j x_j \phi(x)\text{ with }a_j\coloneqq\frac{\int_{\B^3_{r_{\mathrm s}}}x_j f_\chi(x) \d x}{\int_{\B^3_{r_{\mathrm s}}}x_j^2 \phi(x) \d x}\text{ for all }j\in\{1,2,3\}.\]
Since the fact that $\phi$ is a radial function implies that
\[ \int_{\B^3_{r_{\mathrm s}}}x_j x_\ell \phi(x) \d x = 0\text{ for all }j,\ell\in\{1,2,3\} \text{ with }j\neq\ell, \]
we find that
	\[\int_{\B^3_{r_{\mathrm s}}}x_\ell \tilde f_\chi(x) \d x=\int_{\B^3_{r_{\mathrm s}}}x_\ell f_\chi(x) \d x-\sum_{j=1}^3a_j\int_{\B^3_{r_{\mathrm s}}}x_jx_\ell\phi(x)\d x = 0\text{ for all }\ell\in\{1,2,3\}\]
so that the function $\tilde f_\chi$ has, by construction, vanishing first-order moments, which shows that $\tilde f_\chi\in \mathcal{O}$.

Next, we remark that the error introduced by this correction can be estimated by using that the first moment of $f$ is zero by
\[ \norm{f_\chi-\tilde f_\chi}_{L^2} \leq \sum_{j=1}^{3}\abs{a_j}\left(\int_{\B^3_{r_{\mathrm s}}}x_j^2\phi^2(x)\d x\right)^{\frac12} \leq \sum_{j=1}^3 C_j\abs*{\int_{\B^3_{r_{\mathrm s}}}x_j(f_\chi(x)-f(x))\d x} \le C\norm{f_\chi-f}_{L^2}, \]
where we defined the constants
\begin{equation}\label{eq:constant_smooth_approx}
C_j\coloneqq\frac{\left(\int_{\B^3_{r_{\mathrm s}}}x_j^2\phi^2(x)\d x\right)^{\frac12}}{\int_{\B^3_{r_{\mathrm s}}}x_j^2 \phi(x) \d x}\text{ and }C\coloneqq \max\Bigg\{1,r_{\mathrm s}\abs{B_{r_{\mathrm s}}^3}^{\frac12}\sum_{j=1}^3C_j\Bigg\}.
\end{equation}
Using that we have the upper bound from \autoref{eq:ineq1} by our initial choice of $f_\chi$, we end up with
\begin{equation}\label{eq:ineq2}
\norm{f_\chi-\tilde f_\chi}_{L^2}\le\frac\delta3
\end{equation}
To finally ensure that the approximation for $f$ is smooth and retains compact support inside $\B^3_{r_{\mathrm s} - 2\varepsilon}$, we convolve $\tilde f_\chi$ with a radially symmetric mollifier $\varphi_{\tilde\delta}\in C_{\mathrm c}^\infty(\B^3_{r_{\mathrm s}})$ with $\supp(\varphi_{\tilde\delta})\subset \B^3_{\tilde\delta}$ and $\tilde\delta<\frac{\varepsilon}{2}$. The resulting function $g\coloneqq\tilde f_\chi\ast\varphi_{\tilde\delta}$ then satisfies $\supp(g) \subset \B^3_{r_{\mathrm s} - 3\varepsilon + \tilde{\delta}} \subset \B^3_{r_{\mathrm s} - 2\varepsilon}$.

Since the first order moments of $\varphi_{\tilde\delta}$ vanish because of the radial symmetry, we additionally get
	\[\int_{\R^3}x (\tilde f_\chi\ast \varphi_{\tilde\delta})(x)\d x=\int_{\R^3}x \tilde f_\chi(x)\d x \int_{\R^3}\varphi_{\tilde\delta}(x)\d x+\int_{\R^3} \tilde f_\chi(x)\d x \int_{\R^3}x \varphi_{\tilde\delta}(x)\d x=0\]
so that the convolution preserves the vanishing first-order moments of $\tilde f_\chi$.

	By the smoothing property of the convolution, we conclude that $g\in C_{\mathrm c}^\infty(\B^3_{r_{\mathrm s}-2\varepsilon})\cap\mathcal{O}$. And by choosing $\tilde\delta$ sufficiently small, we can further guarantee the error estimate
	\begin{equation}\label{eq:ineq3}
		\norm{\tilde f_\chi-g}_{L^2}\leq\frac{\delta}{3}.
	\end{equation}
	
	By combining the inequalities from \autoref{eq:ineq1}, \autoref{eq:ineq2}, and \autoref{eq:ineq3}, we thus obtain
	\[\norm{f-g}_{L^2}\leq\norm{f-f_\chi}_{L^2}+\norm{f_\chi-\tilde f_\chi}_{L^2}+\norm{\tilde f_\chi-g}_{L^2}\leq\frac{\delta}{3}+\frac{\delta}{3}+\frac{\delta}{3}=\delta.\]
\end{proof}

Later on, we want to construct an asymmetric point set with the maximum projection property in \autoref{dt_prop_maximum} of the proof of \autoref{thm:dt_S_nowheredense}.

\begin{lemma}\label{thm:dt_max_proj}
	Let $\varepsilon\in\big(0,\frac{r_{\mathrm s}}2(1-\cos(\frac\pi6))\big)$. Moreover, let $P\subset\partial \B^3_{r_{\mathrm s}-\varepsilon}$ be a DT-asymmetric point set fulfilling the properties from \autoref{thm:dt_asym_points_existence} and such that there exist for every $v\in S^2$ a points $p_1(v)\in P$ that lies in the spherical cap
			\[C_\varepsilon(v)\coloneqq \{x \in \partial \B^3_{r_{\mathrm s}-\varepsilon} \mid \inner{x}{v} > r_{\mathrm s}-2\varepsilon\}\]
of $\partial\B^3_{r_{\mathrm s}-\varepsilon}$ centered at $v$ with height $\varepsilon$.

	Then, for every choice of directions $\xi,\eta,\nu\in S^2$ with $\inner{\xi}{\eta}=0$, there exists a direction $u\in S^2\cap\spn\{\xi,\eta\}$ such that there is a point $p_1(u)\in P\cap C_\varepsilon(u)\setminus(\spn\{\xi,\eta\}\cup\spn\{\nu\})$ satisfying
		\begin{equation}\label{eq:dt_max_proj}
			\inner{p_1(u)}u>\inner pu\text{ for all }p\in P\setminus\{p_1(u)\}.
		\end{equation}
\end{lemma}
\begin{proof}
	Let $\xi,\eta,\nu\in S^2$ with $\inner{\xi}{\eta}=0$.
	
	According to \autoref{eq:DT_point_exist_cond1} in \autoref{thm:dt_asym_points_existence}, every three points in $P$ are linearly independent. Hence, there can at most two points in $P$ lying in $\spn\{\xi,\eta\}$. We avoid these points by excluding all $u \in S^2 \cap \spn\{\xi, \eta\}$ for which either of these points is contained in $C_\varepsilon(u)$. This prohibits for the choice of $u$ at most two arcs $A_1$ and $A_2$ each with a central angle $2\arccos(\frac{r_{\mathrm s}-2\varepsilon}{r_{\mathrm s}-\varepsilon})<\frac\pi3$ on the circle $S^2 \cap \spn\{\xi, \eta\}$.

Similarly, we exclude for $u$ the directions where we would have $\spn\{\nu\}\cap C_\varepsilon(u)\ne\emptyset$, which disallows at most two additional arcs $A_3$ and $A_4$ with central angles less than $\frac\pi3$.
	
	Assume now that we have for some $u\in S^2 \cap \spn\{\xi, \eta\}$ that $\inner{p_1}{u}= \inner{p_2}{u}$ for some points $p_1,p_2\in P\cap C_\varepsilon(u)$, meaning that \autoref{eq:dt_max_proj} is potentially violated. Furthermore, suppose that these points also fulfill the equality $\inner{p_1}{\tilde u}= \inner{p_2}{\tilde u}$ for every sufficiently small perturbation $\tilde u=\bm{R}_{\xi\times\eta}(\delta)u$ of $u$, where $\bm{R}_{\xi\times\eta}(\delta)\in \SO(3)$ represents the rotation about the vector $\xi\times\eta\in S^2$ by the angle $\delta\in\R$ which shall be chosen such that $p_1,p_2\in P\cap C_\varepsilon(\tilde u)$.
This means that this potential violation of \autoref{eq:dt_max_proj} cannot be resolved by minimally altering the direction.
	
Writing such points $p_1$ and $p_2$ as $p_j=\pi_{\xi\times\eta}(p_j)+\alpha_j\xi\times\eta$ with $\alpha_j\in\R$, $j\in\{1,2\}$, with the orthogonal projection $\pi_{\xi\times\eta}$ onto $\spn\{\xi,\eta\}$ as in \autoref{eq:det_orth_proj}, we get from $u,\tilde u\in\spn\{\xi,\eta\}$ that
\[ \inner{\pi_{\xi\times\eta}(p_1)-\pi_{\xi\times\eta}(p_2)}{u}=\inner{p_1-p_2}u = 0 \text{ and }\inner{\pi_{\xi\times\eta}(p_1)-\pi_{\xi\times\eta}(p_2)}{\tilde u}=\inner{p_1-p_2}{\tilde u}=0, \]
which directly implies that $\pi_{\xi\times\eta}(p_1)=\pi_{\xi\times\eta}(p_2)$.

	However, this can only happen for at most two pairs of points in $P$, since if we had three pairs $(p_{1,k},p_{2,k})_{k=1}^3$ with $\pi_{\xi\times\eta}(p_{1,k})=\pi_{\xi\times\eta}(p_{2,k})$ for all $k\in\{1,2,3\}$, the three planes $\spn\{p_{1,k},p_{2,k}\}$, $k\in\{1,2,3\}$, would intersect in the line with direction $\xi\times\eta$ and it would thus hold that $\det(p_{1,1}\times p_{2,1},p_{1,2}\times p_{2,2},p_{1,3}\times p_{2,3})=0$, contradicting \autoref{eq:DT_point_exist_cond2} of \autoref{thm:dt_asym_points_existence}.
	
	Excluding the directions $u\in S^2 \cap \spn\{\xi, \eta\}$ for which such a pair of points is contained in $C_\varepsilon(u)$, which by our choice of $\varepsilon$ are again two circular arcs $A_5,A_6\subset S^2 \cap \spn\{\xi, \eta\}$ with central angles less than $\frac\pi3$, we find in every neighborhood of every direction in the remaining set $S^2 \cap \spn\{\xi, \eta\}\setminus(\bigcup_{\ell=1}^6\overline{A_\ell})$ a direction $u\in S^2 \cap \spn\{\xi, \eta\}$ such that there exists a point $p_1(u)\in P\cap C_\varepsilon(u)\setminus(\spn\{\xi,\eta\}\cup\spn\{\nu\})$ fulfilling \autoref{eq:dt_max_proj}.
\end{proof}

For the final argument in the proof of \autoref{thm:dt_S_nowheredense}, we use some properties of a holomorphic extension of the function $H\colon[-k_0,k_0]\to\C^3$, $H(\mu)\coloneqq \mu\xi+h(\mu)\eta$, with $h$ being defined in \autoref{eq:dt_def_h} for some orthogonal unit vectors $\xi,\eta\in S^2$.

\begin{lemma}\label{thm:dt_prop_H}
	Let $\xi,\eta\in S^2$ be orthogonal unit vectors and $\Gamma\coloneqq(-\infty,-k_0)\cup(k_0,\infty)$. Then the function 
	\begin{equation}\label{eq:dt_def_H_polar}
			H\colon\C\setminus\Gamma\to\C^3,\;H(z)\coloneqq z\xi+\left(\sqrt{k_0^2-z^2}-k_0\right)\eta,
	\end{equation}
	fulfills for every $z=r\e^{\i\varphi}\in\C\setminus\Gamma$, $r\in[0,\infty)$, $\varphi\in(-\pi,\pi)\setminus\{0\}$, that
	\begin{enumerate}
	\item \label{item:dt_prop_H1}
	$\displaystyle\norm{H(z)}\geq\abs z$ and $\norm{\Im(H(z))}\geq \abs{\Im(z)}$,
	\item \label{item:dt_prop_H4}
	$\displaystyle\lim_{z\to\infty}\frac1z\sgn(\Im(z))\sum_{j=1}^3H_j^2(z)=2\i k_0$, and
	\item \label{item:dt_prop_H2}
	$\displaystyle\Im(H(r\e^{\i\varphi}))=r\sin(\varphi)\xi-\sgn(\sin(2\varphi))\sqrt{\frac{1}{2}\left(\sqrt{k_0^4-2 k_0^2r^2\cos(2 \varphi)+r^4}-k_0^2+r^2\cos(2\varphi)\right)}\,\eta$.
	\end{enumerate}
\end{lemma}
\begin{proof}\mbox{}\\[-3ex]
\begin{enumerate}
\item
Since $\xi$ and $\eta$ are orthonormal, we get
\begin{align*}
&\norm{H(z)}^2 = \abs{z}^2+\abs*{\sqrt{k_0^2-z^2}-k_0}^2 \ge \abs z^2\text{ and} \\
&\norm{\Im(H(z))}^2 = (\Im(z))^2+\left(\Im\left(\sqrt{k_0^2-z^2}-k_0\right)\right)^2 \ge (\Im(z))^2.
\end{align*}

\item
Again, the orthonormality of $\xi$ and $\eta$ gives us
\begin{align*}
\lim_{z\to\infty}\frac1z\sgn(\Im(z))\sum_{j=1}^3H_j^2(z) &= \lim_{z\to\infty}\frac1z\sgn(\Im(z))\left(2k_0^2-2k_0\sqrt{k_0^2-z^2}\right) \\
&= -2k_0\lim_{z\to\infty}\sgn(\Im(z))\frac{\sqrt{-z^2}}z = 2\i k_0.
\end{align*}

\item
Applying the formula for the imaginary part of the square root of a complex number, as given in \cite[Section 3.7.27]{AbrSte72}, to the term $\sqrt{k_0^2-r^2 \e^{2\i\varphi}}$ in \autoref{eq:dt_def_H_polar}, we directly obtain the predicted expression for $\Im(H(r\e^{\i\varphi}))$.
\end{enumerate}
\end{proof}

At last, we are interested in the behavior of the normalized vector $z\mapsto\frac{\Im(H(z))}{\norm{\Im(H(z))}}$ in the direction of the imaginary part of $H$.

\begin{lemma}\label{thm:dt_prop_HIm}
Let $\xi,\eta\in S^2$ be orthogonal unit vectors and the function $H$ be given by \autoref{eq:dt_def_H_polar}.

We then have for every $r\in(k_0,\infty)$ that the function
\[ H_r\colon(-\pi,\pi)\setminus\{0\}\to S^2\cap\spn\{\xi,\eta\}\setminus\{-\eta,\eta\},\;H_r(\varphi)\coloneqq\frac{\Im(H(r\e^{\i\varphi}))}{\norm{\Im(H(r\e^{\i\varphi)}))}}, \]
is well-defined and surjective.

Moreover, we have the asymptotic behavior
\begin{equation}\label{eq:dt_prop_HIm_asymptotics}
\lim_{r\to\infty}\sup_{\varphi\in(-\pi,\pi)\setminus\{0\}}\norm{H_r(\varphi)-\sin(\varphi)\xi-\sgn(\varphi)\cos(\varphi)\eta} = 0.
\end{equation}
\end{lemma}

\begin{proof}
According to \autoref{item:dt_prop_H2} of \autoref{thm:dt_prop_H}, we can write for every $r>k_0$ and $\varphi\in(-\pi,\pi)\setminus\{0\}$
\[ H_r(\varphi)=a_1(r,\varphi)\xi-a_2(r,\varphi)\eta \]
with the functions
\begin{align*}
&\ell(r,\varphi)\coloneqq\sqrt{\frac{1}{2}\left(\sqrt{k_0^4-2 k_0^2r^2\cos(2 \varphi)+r^4}-k_0^2+r^2\cos(2\varphi)\right)},\\
&a_1(r,\varphi)\coloneqq\frac{r\sin(\varphi)}{\sqrt{r^2\sin^2(\varphi)+\ell^2(r,\varphi)}},\;\text{and}\\
&a_2(r,\varphi)\coloneqq\sgn(\sin(2\varphi))\frac{\ell(r,\varphi)}{\sqrt{r^2\sin^2(\varphi)+\ell^2(r,\varphi)}}.
\end{align*}
In particular, we see that $H_r(\varphi)\in\spn\{\xi,\eta\}$, and, by definition, we also have $\norm{H_r(\varphi)}=1$, so that it consequently holds that $H_r$ indeed maps into $S^2\cap\spn\{\xi,\eta\}$.

Moreover, for every $r\in(k_0,\infty)$, the functions $\varphi\mapsto a_j(r,\varphi)$, $j\in\{1,2\}$, are well-defined and continuous with respect to $\varphi$ on $(-\pi,0)$ and $(0,\pi)$. The continuity of $a_1$ is straightforward and the only potential discontinuity of $a_2$ would be at $\varphi\in\{-\frac{\pi}{2},\frac\pi2\}$, where $\varphi\mapsto\sgn(\sin(2\varphi))$ has a jump discontinuity. However, since $\ell(r, \frac{\pi}{2}) = \ell(r,-\frac\pi2) = 0$, there is no jump and $a_2$ is continuous on both intervals.

Since we have $a_1(r,-\varphi)<0$ and $a_1(r,\varphi)>0$ for all $\varphi\in(0,\pi)$, it only remains to verify that $\varphi\mapsto a_2(r,\varphi)$ maps $(-\pi,0)$ and $(0,\pi)$ surjectively onto $(-1,1)$ to prove that $H_r(\varphi)$ maps for every $r\in(k_0,\infty)$ onto $S^2\cap\spn\{\xi,\eta\}\setminus\{-\eta,\eta\}$. It is thus enough to remark that we have
\begin{alignat*}{2}
&\lim_{\varphi\downarrow-\pi}a_2(r,\varphi)=1&&\text{ and }\lim_{\varphi\uparrow0}a_2(r,\varphi)=-1\text{ as well as} \\
&\lim_{\varphi\downarrow0}a_2(r,\varphi)=1&&\text{ and }\lim_{\varphi\uparrow\pi}a_2(r,\varphi)=-1.
\end{alignat*}
The restriction $r > k_0$ is necessary since for $r \le k_0 , \ell(r,\varphi) \to 0$ as $\varphi \to 0$, causing $|a_2|$ to potentially not approach 1.

Observing further that we get from the trigonometric identity $\frac12(1+\cos(2\varphi))=\cos^2(\varphi)$ that
\[ \lim_{r\to\infty}\sup_{\varphi\in(-\pi,\pi)\setminus\{0\}}\abs*{\frac{\ell(r,\varphi)}r-\abs{\cos(\varphi)}} = 0, \]
we find that
\begin{align*}
&\lim_{r\to\infty}\sup_{\varphi\in(-\pi,\pi)\setminus\{0\}}\abs*{a_1(r,\varphi)-\sin(\varphi)} = 0\text{ and} \\
&\lim_{r\to\infty}\sup_{\varphi\in(-\pi,\pi)\setminus\{0\}}\abs*{a_2(r,\varphi)-\sgn(\varphi)\cos(\varphi)} = \lim_{r\to\infty}\sup_{\varphi\in(-\pi,\pi)\setminus\{0\}}\abs[\big]{a_2(r,\varphi)-\sgn(\sin(2\varphi))\abs{\cos(\varphi)}} = 0,
\end{align*}
which proves \autoref{eq:dt_prop_HIm_asymptotics}.
\end{proof}

\subsection*{Acknowledgments}
This research was funded in whole or in part by the Austrian Science Fund (FWF)
SFB 10.55776/F68 ``Tomography Across the Scales'', project F6804-N36
(Quantitative Coupled Physics Imaging). For open access purposes, the authors have
applied a CC BY public copyright license to any author-accepted manuscript
version arising from this submission.

\section*{References}
\renewcommand{\i}{\ii}
\printbibliography[heading=none]

\end{document}